\documentclass[twoside,leqno,10pt, A4]{amsart}
\usepackage{amsfonts}
\usepackage{amsmath}
\usepackage{amscd}
\usepackage{amssymb}
\usepackage{amsthm}
\usepackage{amsrefs}
\usepackage{latexsym}
\usepackage{mathrsfs}
\usepackage{bbm}
\usepackage{enumerate}
\usepackage{graphicx}

\usepackage{amsfonts}
\usepackage{amsmath}
\usepackage{amscd}
\usepackage{amssymb}
\usepackage{amsthm}
\usepackage{amsrefs}
\usepackage{latexsym}
\usepackage{mathrsfs}
\usepackage{bbm}
\usepackage{amscd}
\usepackage{amssymb}
\usepackage{amsthm}
\usepackage{amsrefs}
\usepackage{latexsym}
\usepackage{mathrsfs}
\usepackage{bbm}
\usepackage{enumerate}
\usepackage{graphicx}
\usepackage{color}
\begin{document}

\newtheorem{theorem}[subsection]{Theorem}
\newtheorem{proposition}[subsection]{Proposition}
\newtheorem{lemma}[subsection]{Lemma}
\newtheorem{corollary}[subsection]{Corollary}
\newtheorem{conjecture}[subsection]{Conjecture}
\newtheorem{prop}[subsection]{Proposition}
\numberwithin{equation}{section}
\newcommand{\mr}{\ensuremath{\mathbb R}}
\newcommand{\mc}{\ensuremath{\mathbb C}}
\newcommand{\dif}{\mathrm{d}}
\newcommand{\intz}{\mathbb{Z}}
\newcommand{\ratq}{\mathbb{Q}}
\newcommand{\natn}{\mathbb{N}}
\newcommand{\comc}{\mathbb{C}}
\newcommand{\rear}{\mathbb{R}}
\newcommand{\prip}{\mathbb{P}}
\newcommand{\uph}{\mathbb{H}}
\newcommand{\fief}{\mathbb{F}}
\newcommand{\majorarc}{\mathfrak{M}}
\newcommand{\minorarc}{\mathfrak{m}}
\newcommand{\sings}{\mathfrak{S}}
\newcommand{\fA}{\ensuremath{\mathfrak A}}
\newcommand{\mn}{\ensuremath{\mathbb N}}
\newcommand{\mq}{\ensuremath{\mathbb Q}}
\newcommand{\half}{\tfrac{1}{2}}
\newcommand{\f}{f\times \chi}
\newcommand{\summ}{\mathop{{\sum}^{\star}}}
\newcommand{\chiq}{\chi \bmod q}
\newcommand{\chidb}{\chi \bmod db}
\newcommand{\chid}{\chi \bmod d}
\newcommand{\sym}{\text{sym}^2}
\newcommand{\hhalf}{\tfrac{1}{2}}
\newcommand{\sumstar}{\sideset{}{^*}\sum}
\newcommand{\sumprime}{\sideset{}{'}\sum}
\newcommand{\sumprimeprime}{\sideset{}{''}\sum}
\newcommand{\sumflat}{\sideset{}{^\flat}\sum}
\newcommand{\shortmod}{\ensuremath{\negthickspace \negthickspace \negthickspace \pmod}}
\newcommand{\V}{V\left(\frac{nm}{q^2}\right)}
\newcommand{\sumi}{\mathop{{\sum}^{\dagger}}}
\newcommand{\mz}{\ensuremath{\mathbb Z}}
\newcommand{\leg}[2]{\left(\frac{#1}{#2}\right)}
\newcommand{\muK}{\mu_{\omega}}
\newcommand{\thalf}{\tfrac12}
\newcommand{\lp}{\left(}
\newcommand{\rp}{\right)}
\newcommand{\Lam}{\Lambda_{[i]}}
\newcommand{\lam}{\lambda}
\def\L{\fracwithdelims}
\def\om{\omega}
\def\pbar{\overline{\psi}}
\def\phi{\varphi}
\def\lam{\lambda}
\def\lbar{\overline{\lambda}}
\newcommand\Sum{\Cal S}
\def\Lam{\Lambda}

\theoremstyle{plain}
\newtheorem{conj}{Conjecture}
\newtheorem{remark}[subsection]{Remark}

\makeatletter
\def\widebreve{\mathpalette\wide@breve}
\def\wide@breve#1#2{\sbox\z@{$#1#2$}%
     \mathop{\vbox{\m@th\ialign{##\crcr
\kern0.08em\brevefill#1{0.8\wd\z@}\crcr\noalign{\nointerlineskip}%
                    $\hss#1#2\hss$\crcr}}}\limits}
\def\brevefill#1#2{$\m@th\sbox\tw@{$#1($}%
  \hss\resizebox{#2}{\wd\tw@}{\rotatebox[origin=c]{90}{\upshape(}}\hss$}
\makeatletter

\title[Real zeros of quadratic Hecke L-functions]{Real zeros of quadratic Hecke $L$-functions}

\author{Peng Gao}
\address{School of Mathematical Sciences, Beihang University, Beijing 100191, P. R. China}
\email{penggao@buaa.edu.cn}
\begin{abstract}
 We study real zeros of a family of quadratic Hecke $L$-functions in the Gaussian field to show that more than twenty percent of the members have no zeros on the interval $(0, 1]$.
\end{abstract}

\maketitle



\section{Introduction}
\label{sec 1}

  The non-vanishing issue of central values of $L$-functions has been studied intensively in the literature due to its deep arithmetic implications.  For the classical case, it is expected that the corresponding $L(s, \chi)$ never vanishes at
$s=1/2$ for any Dirichlet character $\chi$. This statement first appeared as a conjecture of S. Chowla \cite{Chow}, who concerned with the special case of primitive real characters.

  Partial resolutions to Chowla's conjecture was first given by M. Jutila \cite{Jutila},  who evaluated the first two moments of the family of quadratic Dirichlet $L$-functions to show that infinitely many such $L$-functions do not vanish at the central value.  By evaluating mollified moments instead, K. Soundararajan \cite{sound1} improved the result of Jutila to show that unconditionally, at least $87.5\%$ of the members of the quadratic family have non-vanishing central values. This percentage can be further improved to be $94.27\%$ if one assumes the generalized Riemann hypothesis (GRH). In fact, this follows from a result of A. E. \"Ozluk and C. Snyder \cite{O&S} on the one-level density of low-lying zeros of the family of quadratic Dirichlet $L$-functions, replacing the test function used in \cite{O&S} by an optimal one given in a paper of H. Iwaniec, W. Luo and P. Sarnak \cite[Appendix A]{ILS} .

  The statements of GRH and Chowla's conjecture imply that $L(\sigma, \chi) \neq 0$ for every primitive quadratic Dirichlet character
and all  $0 \leq \sigma \leq 1$. However, it was not even known previously whether every such Dirichlet $L$-function of sufficiently large conductor has a non-trivial real zero or not until J. B. Conrey and K. Soundararajan \cite{C&S} proved in 2002 that for at least $20\%$ of the odd square-free integers $d \geq 0$ we have $L(\sigma, \chi_{-8d}) >$ 0 for $0 \leq \sigma \leq 1$, where $\chi_{-8d}=\leg {-8d}{\cdot}$ is the Kronecker symbol.

  Motivated by the above result of Conrey and Soundararajan, we investigate in this paper the real zeros of a family of quadratic Hecke $L$-functions in the Gaussian number field  $\mq(i)$ . To state our result, we first introduce some notations.  Throughout the paper, we denote $K=\mq(i)$ and $\mathcal{O}_K$ for the ring of integers $\mz[i]$ of $K$. For any $c \in K$, we write $N(c)$ for its norm. We denote $L(s,\chi), \zeta_{K}(s)$ for the $L$-function associated to a Hecke character $\chi$ of $K$ and the Dedekind zeta function of $K$, respectively.

  We say a Hecke character $\chi$ is primitive modulo $q$ if it does not factor through $\left (\mathcal{O}_K / (q') \right )^{\times}$ for any  divisor $q'$ of $q$ such that $N(q')<N(q)$. We also say that $\chi$ is of trivial infinite type if its component at infinite places of $K$ is trivial.  For any $c \in \mathcal O_K$, we reserve in this paper the symbol $\chi_c$ for the quadratic residue symbol $\leg {c}{\cdot}$ defined in Section \ref{sec2.4}. It is shown in \cite[Section 2.1]{Gao2} that $\chi_{(1+i)^5d}$ defines a primitive quadratic Hecke character modulo $(1+i)^5d$ of trivial infinite type for any odd, square-free $d \in \mathcal{O}_K$. Here, we say that $d$ is odd if $(d,2)=1$ and $d$ is square-free if the ideal $(d)$ is not divisible by the square of any prime ideal.  The family of $L$-functions that we consider in this paper is given by
\begin{align*}
  \mathcal F = \big\{ L(s, \chi_{(1+i)^5d}) :  d \in \mathcal O_K \hspace{0.05in} \text{odd, square-free} \big\}.
\end{align*}

  Notice that every $L$-function in the above family satisfies a functional equation given by \eqref{fneqnquad} which implies that $\Gamma(s)L(s, \chi_{(1+i)^5d})$ is entire on the whole complex plane. As $\Gamma(s)$ has a simple pole at $s=0$, this implies that $L(0, \chi_{(1+i)^5d})=0$. Thus, if one assumes GRH for the above family of $L$-functions as well as the non-vanishing of these $L$-functions at the central point $1/2$, then one can deduce that $L(\sigma, (1+i)^5d) \neq 0$ for all $0< \sigma \leq 1$. Consequently, we have $L(\sigma, (1+i)^5d) > 0$ for all $0< \sigma \leq 1$ by continuity and the fact that $L(2, (1+i)^5d) > 0$. In this paper, we show that unconditionally, at least a positive percent of these $L$-functions do not vanish on $(0, 1]$. By further keeping in mind that the number of odd, square-free $d \in \mathcal O_K$ such that $N(d) \leq x$ is asymptotically $\frac{2\pi x}{3\zeta_K(2)}$  (see \cite[Section 3.1]{G&Zhao4}), our result is given as follows.
\begin{theorem}
\label{thm: nonvanishing}
For at least $20\%$ of the members of $L$-functions in $\mathcal F$,  we have $L(\sigma,\chi_{(1+i)^5d}) >0$
for $0< \sigma \le 1$.  More precisely, the
number of $L$-functions in $\mathcal F$
such that $N(d) \le x$ and $L(\sigma,\chi_{(1+i)^5d}) >0$ for all $0< \sigma \le 1$
exceeds $\frac{1}{5} (\frac {2\pi x}{3\zeta_K(2)})$ for all large $x$.
\end{theorem}

  Our proof of Theorem \ref{thm: nonvanishing} is mainly based on the approach of Conrey and Soundararajan in their above mentioned work  \cite{C&S} concerning the non-vanishing of quadratic Dirichlet $L$-functions on the real line. We also note here that similar to what is pointed out in \cite{C&S}, the approach in our paper implies that for all large $x$, the number of odd, square-free $d \in \mathcal O_K$ with $N(d) \le x$ such
that $L(s,\chi_{(1+i)^5d})$ has a zero in $[\sigma,1]$ is $\ll x^{1-(1-\varepsilon)(\sigma-\frac 12)}$
for any fixed $\sigma \geq 1/2$ and $\varepsilon >0$. Another outcome of our approach is that there are many $L$-functions having no non-trivial zeros in a thin rectangle containing the real line.  More precisely, there exists a constant $c>0$ such that for at least $20\%$ of the members of $\mathcal F$ with $N(d) \le x$, the corresponding $L(s,\chi_{(1+i)^5d})$
has no zeros in the rectangle $\{\sigma+it: \ \ \sigma \in (0,1], \ \ |t|\le c/\log x\}$ for all large $x$.

\section{Preliminaries}
\label{sec 2}

 The proof of our result requires many tools as well as auxiliary results, which we gather them here first.
\subsection{Quadratic residue symbol and Gauss sum}
\label{sec2.4}
  It is well-known that $K=\mq(i)$ has class number one and that every ideal in $\mathcal{O}_K$ co-prime to $2$ has a unique generator which is $\equiv 1 \bmod (1+i)^3$. Such a generator is called primary.  Further notice that $(1+i)$ is the only prime ideal in $\mathcal{O}_K$ that lies above the ideal $(2) \in \mz$. We now fix a generator $n$ for each ideal $(n)$ in $\mathcal{O}_K$ by taking $n$ to be of the form $(1+i)^m n'$ with $m \geq 1$ and $n'$ primary.  We denote the set of these generators by $G$ throughout the paper. For $a, b \in \mathcal O_K$, we denote $(a,b)$ for their greatest common divisor such that $(a, b) \in G$.

  We denote $\varpi$ for a prime element in $\mathcal O_K$, which means that the ideal generated by $\varpi$ is a prime ideal. We denote the group of units in $\mathcal{O}_K$ by $U_K$ and the discriminant of $K$ by $D_{K}$. Thus, we have $U_K=\{ \pm 1, \pm i \}$ and $D_{K}=-4$. We write   $d = \square$ for a perfect square in $\mathcal O_K$, by which we mean that $d=n^2$ for some $n \in \mathcal O_K$.

  The quadratic residue symbol $\leg{\cdot}{\varpi}$ is defined for any odd prime $\varpi$ such that for any $a \in
\mathcal{O}_{K}$,  $\leg{a}{\varpi} =0$ when $\varpi | a$ and $\leg{a}{\varpi} \equiv
a^{(N(\varpi)-1)/2} \pmod{\varpi}$ with $\leg{a}{\varpi} \in \{
\pm 1 \}$ when $(a, \varpi)=1$.
 The definition is then extended to $\leg{\cdot}{n}$ multiplicatively for any odd $n \in \mathcal O_K$. Here, we define $\leg {\cdot}{n}=1$ for $n \in U_K$.

  As mentioned in Section \ref{sec 1}, we denote $\chi_{c}$ for the quadratic residue symbol $\leg {c}{\cdot}$.  In this paper,  we regard $\chi_{\pm 1}$ as a principal character modulo $1$, so that $\chi_{\pm 1}(a)=1$ for all $a \in \mathcal O_K$. Note that this implies that $L(s, \chi_{\pm 1})=\zeta_K(s)$. For other values of $c$, we shall regard $\chi_{c}$ as a Hecke character of trivial infinite type modulo $16c$ as this is justified in \cite[Section 2.1]{G&Zhao2019}. In particular, we notice that $\leg {c}{a}=0$ when $1+i|a$ and $c \neq \pm 1$.

 We define the quadratic Gauss sum $g(r, n)$ for $r, n \in \mathcal O_K$ with $n$ being odd, by
\begin{align*}
  g(r,n) = \sum_{x \bmod{n}} \leg{x}{n} \widetilde{e}\leg{rx}{n},
\end{align*}
  where we define for any complex number $z$,
\begin{align*}
 \widetilde{e}(z) =\exp \left( 2\pi i  \left( \frac {z}{2i} - \frac {\bar{z}}{2i} \right) \right) .
\end{align*}

 We denote $\mu_{[i]}$  for the analogue on $\mathcal O_K$ of the usual M\"obius function on $\mz$ and $\varphi_{[i]}(n)$ for the number of elements in the reduced residue class of $\mathcal{O}_K/(n)$. We recall the following explicitly evaluations of $g(r,n)$ given in \cite[Lemma 2.2]{G&Zhao4}.
\begin{lemma} \label{Gausssum}
\begin{enumerate}[(i)]
\item  We have
\begin{align*}
 g(rs,n) & = \overline{\leg{s}{n}} g(r,n), \qquad (s,n)=1, \\
   g(k,mn) & = g(k,m)g(k,n),   \qquad  m,n \text{ primary and } (m , n)=1 .
\end{align*}
\item Let $\varpi$ be a primary prime in $\mathcal{O}_K$. Suppose $\varpi^{h}$ is the largest power of $\varpi$ dividing $k$. (If $k = 0$ then set $h = \infty$.) Then for $l \geq 1$,
\begin{align*}
g(k, \varpi^l)& =\begin{cases}
    0 \qquad & \text{if} \qquad l \leq h \qquad \text{is odd},\\
    \varphi_{[i]}(\varpi^l) \qquad & \text{if} \qquad l \leq h \qquad \text{is even},\\
    -N(\varpi)^{l-1} & \text{if} \qquad l= h+1 \qquad \text{is even},\\
    \leg {ik\varpi^{-h}}{\varpi}N(\varpi)^{l-1/2} \qquad & \text{if} \qquad l= h+1 \qquad \text{is odd},\\
    0 \qquad & \text{if} \qquad l \geq h+2.
\end{cases}
\end{align*}
\end{enumerate}
\end{lemma}

\subsection{The approximate functional equation}
\label{sect: apprfcneqn}
    For any primitive quadratic Hecke character $\chi$ of $K$ of trivial infinite type, a well-known result of E. Hecke asserts that $L(s, \chi)$ has an analytic continuation to the entire complex plane with a simple pole at $s=1$ only when $\chi$ is principal. Moreover,  $L(s, \chi)$ satisfies the following
functional equation (see \cite[Theorem 3.8]{iwakow})
\begin{align}
\label{fneqn}
  \Lambda(s, \chi) = W(\chi)(N(m))^{-1/2}\Lambda(1-s, \chi),
\end{align}
   where $m$ is the conductor of $\chi$, $|W(\chi)|=(N(m))^{1/2}$ and
\begin{align}
\label{Lambda}
  \Lambda(s, \chi) = (|D_K|N(m))^{s/2}(2\pi)^{-s}\Gamma(s)L(s, \chi).
\end{align}

   In particular, we have the following functional equation for $\zeta_K(s)$:
\begin{align}
\label{fcneqnforzeta}
\pi^{-s}\Gamma(s)\zeta_K(s)=\pi^{-(1-s)}\Gamma(1-s)\zeta_K(1-s).
\end{align}

   When $\chi=\chi_{(1+i)^5d}$ for any odd, square-free $d \in \mathcal{O}_K$, we combine \cite[Theorem 3.8]{iwakow} and \cite[Lemma 2.2]{Gao2} to see that
$W(\chi_{(1+i)^5d})=g(\chi_{(1+i)^5d})=\sqrt{N((1+i)^5d)}$. Thus, the functional equation \eqref{fneqn} in this case becomes
\begin{align}
\label{fneqnquad}
  \Lambda(s, \chi_{(1+i)^5d}) = \Lambda(1-s, \chi_{(1+i)^5d}).
\end{align}

 We now set
\begin{align}
\label{xidef}
\xi(s,\chi_{(1+i)^5d})=\biggl(\frac{2^5N(d)}{\pi^2}\biggr)^{\frac{s}{2}-\frac 14}
\Gamma(s) L(s,\chi_{(1+i)^5d}).
\end{align}
   It follows from \eqref{Lambda} and \eqref{fneqnquad} that we have the following functional equation
\begin{align}
\label{fcneqnforxi}
\xi(s,\chi_{(1+i)^5d})=\xi(1-s,\chi_{(1+i)^5d}).
\end{align}

 Throughout the paper, we fix a positive real number $\kappa \leq 1/100$ and let $\delta_1$ and $\delta_2$ be two complex numbers satisfying $\max(|\delta_1|, |\delta_2)|) \leq \kappa$. We further denote $\tau= \frac{\delta_1+\delta_2}2$, $\delta =
\frac{\delta_1-\delta_2}{2}$ so that we have  $\max(|\tau|, |\delta)|) \leq \kappa$ as well. We define for real numbers $x >0, c > |\Re(\tau)|$,
\begin{align}
\label{Wdt}
  W_{\delta,\tau}(x) = \frac{1}{2\pi i} \int\limits_{(c)} \Gamma_\delta(s)
x^{-s} \frac{2s}{s^2 -\tau^2} ds,
\end{align}
 where
$$
\Gamma_\delta(s) = \Gamma\Big(\frac 12 +s +\delta \Big)
\Gamma\Big(\frac 12 +s -\delta\Big).
$$

  The following Lemma establishes some analytical properties concerning $W_{\delta,\tau}(x)$.
\begin{lemma}
\label{lemma 3.1}  The function $W_{\delta,\tau}(x)$
is a smooth complex-valued function for $x \in (0,\infty)$.  For $x$ near $0$, we
have
\begin{align}
\label{Wxsmall}
W_{\delta,\tau}(x) = \Gamma_\delta(\tau)x^{-\tau} +\Gamma_\delta(-\tau)
x^{\tau} + O(x^{1/4-\epsilon}).
\end{align}
For large $x$ and any integer $\nu$, we have
\begin{align}
\label{Wxlarge}
W_{\delta, \tau}^{(\nu)} (x) \ll_{\nu} x^{\nu+3} e^{-2x^{1/2}}
\ll_{\nu} e^{-x^{1/2}}.
\end{align}
\end{lemma}
\begin{proof}  We move the line of integration in \eqref{Wdt}
to $\Re(s) = -1/4+\epsilon$ to deduce readily \eqref{Wxsmall}.  On the other hand, for any $c>|\Re (\tau)|$, we have
$$
W_{\delta,\tau}^{(\nu)}(x)=\frac{(-1)^{\nu}}{2\pi i} \int\limits_{(c)} \Gamma_\delta (s) s(s+1)\cdots
(s+\nu-1) x^{-s-\nu} \frac{2s}{s^2 -\tau^2} ds.
$$
  This implies that $W_{\delta,\tau}(x)$ is smooth.
To prove \eqref{Wxlarge}, we may suppose that
$x > \nu +4$. Using the facts that $|\Gamma(x+iy)|\le \Gamma(x)$ for $x\ge 1$ and
$s\Gamma(s)=\Gamma(s+1)$, we see that
$$
W_{\delta,\tau}^{(\nu)}(x) \ll_{\nu} |\Gamma(c+\nu+3)|^2 x^{-c} \int\limits_{(c)}
\frac{|ds|}{|s^2-\tau^2|} \ll_{\nu} \Gamma(c +\nu+3)^2 \frac{x^{-c}}
{c-|\Re(\tau)|}\ll_{\nu} \Big( \frac{c+\nu+3}{e}\Big)^{2c+2\nu+6} \frac{x^{-c}}
{c-|\Re(\tau)|},
$$
where the last estimation above follows from Stirling's formula (see \cite[(5.113)]{iwakow}).
The last assertion of the lemma now follows by taking $c= x^{1/2} -\nu-3 (\ge 2)$ above.
\end{proof}

For any complex number $s$ and any primary $n \in \mathcal O_K$,  we define
$$
r_s(n) = \sum_{\substack{ab=n \\ a, b \equiv 1 \bmod (1+i)^3} } \Big( \frac {N(a)}{N(b)}\Big)^s.
$$
 Note that $r_s(n)$ is easily seen to be an even function of $s$. We also define for any odd, square-free $d \in \mathcal O_K$,
$$
A_{\delta,\tau} (d) = \sum_{n \equiv 1 \bmod (1+i)^3} \frac{r_{\delta}(n)}{\sqrt{N(n)}}
\leg{(1+i)^5d}{n} W_{\delta,\tau}\leg{\pi^2 N(n) }{2^5N(d)}.
$$

 Now, we present an approximate functional equation for $
\xi(\frac 12+\delta_1,\chi_{(1+i)^5d})\xi(\frac 12+\delta_2,\chi_{(1+i)^5d})$.
\begin{lemma}
\label{Lemma 3.2} With the notations above for $\delta_1$ and $\delta_2$. We have for any odd, square-free $d \in \mathcal O_K$,
$$
\xi(\tfrac 12+\delta_1,\chi_{(1+i)^5d})\xi(\tfrac 12+\delta_2,\chi_{(1+i)^5d}
) = A_{\delta,\tau}(d).
$$
\end{lemma}
\begin{proof}  We begin with the following integral for some $3/2 - |\Re(\delta)| > c > 1/2+|\Re(\delta)|$,
$$
\frac{1}{2\pi i}\int\limits_{(c)} \xi(\tfrac 12+\delta+s,\chi_{(1+i)^5d})
\xi(\tfrac 12-\delta+s,\chi_{(1+i)^5d}) \frac{2s}{s^2 -\tau^2} ds.
$$
 We evaluate the above integral by first writing $L(\tfrac 12+\delta+s,\chi_{(1+i)^5d}) L(\tfrac 12-\delta+s,\chi_{(1+i)^5d})$
as a  Dirichlet series
$\displaystyle{\sum_{n \equiv 1 \bmod (1+i)^3}} \frac{ r_\delta(n)}{N(n)^{\frac 12+s}} \leg{(1+i)^5d}{n}$ to see that the above expression equals
$A_{\delta,\tau}(d)$ upon integrating term by term.
On the other hand, we move the line of integration to $\Re(s)=-c$ to encounter
poles at $s=\tau$, $-\tau$. The residues contribute
$\xi(\frac 12+\delta+\tau,\chi_{(1+i)^5d})\xi(\frac 12-\delta+\tau,\chi_{(1+i)^5d})
+\xi(\frac 12+\delta-\tau,\chi_{(1+i)^5d})\xi(\frac 12-\delta-\tau,\chi_{(1+i)^5d})
= 2\xi(\frac 12+\delta_1,\chi_{(1+i)^5d})\xi(\frac 12+\delta_2,\chi_{(1+i)^5d})$ by the functional equation \eqref{fcneqnforxi}.
Using \eqref{fcneqnforxi} again, we see that the remaining integral on the $-c$ line equals to $-A_{\delta,\tau}(d)$ via a change of variable $s \to -s$. This leads to the desired result.
\end{proof}

\subsection{Quadratic large sieves}
  We include in this section two large sieve results concerning quadratic Hecke characters. They are generalizations in $K$ of the well-known large sieve results due to D. R. Heath-Brown \cite{DRHB} on quadratic Dirichlet characters. The first lemma can be obtained by applying a large sieve result of K. Onodera \cite{Onodera} on quadratic residue symbols in $K$ in the proof of \cite[Corollary 2]{DRHB} and \cite[Lemma 2.4]{sound1}.
\begin{lemma}\label{lem: estimates for character sums}
  Let $N, Q$ be positive integers, and let $a_1, \cdots, a_n$ be arbitrary complex numbers. Let $S(Q)$ denote the set of $\chi_m$ for square-free $m$ satisfying $N(m) \leq Q$. Then for any $\epsilon > 0$,
\begin{align*}
 \sum_{\chi \in S(Q)}\Big | \sum_{\substack{n \equiv 1 \bmod {(1+i)^3} \\ N(n) \leq N}} a_n\chi(n) \Big |^2 \ll_{\epsilon}  (QN)^{\epsilon}(Q + N) \sum_{\substack{n_1, n_2 \equiv 1 \bmod {(1+i)^3} \\ N(n_1), N(n_2) \leq N \\ n_1n_2=\square}}|a_{n_1}a_{n_2}|.
\end{align*}
 Let $M$ be a positive integer, and for each $m \in \mathcal O_K$ satisfying $N(m) \leq M$, we write $m = m_1m_2^2$ with $m_1$ square-free and $m_2 \in G$. Suppose the sequence $a_n$ satisfies $|a_n| \ll N(n)^\varepsilon$, then
\begin{align*}
\sum_{N(m) \leq M} \frac{1}{N(m_2)} \left|\sum_{N(n) \leq N} a_n \left( \frac{m}{n}\right) \right|^2 \ll (MN)^\varepsilon N (M+N).
\end{align*}
\end{lemma}

  Similarly, combining the above result of Onodera with the derivation of \cite[Theorem 2]{DRHB} and \cite[Lemma 2.5]{sound1}, we have the following result.
\begin{lemma}
\label{lem:2.3}
 Let $S(Q)$ be as in Lemma \ref{lem: estimates for character sums}. For any complex number $\sigma+it$ with $\sigma \geq \frac{1}{2}$, we have
\begin{align*}
\sum_{\chi \in S(Q)} |L(\sigma+it,\chi_{(1+i)^5d})| ^4
\ll& Q^{1+\varepsilon} (1+|t|^2)^{1+\varepsilon} \\
\sum_{\chi \in S(Q)} |L(\sigma+it,\chi_{(1+i)^5d})| ^2
\ll& Q^{1+\varepsilon} (1+|t|^2)^{1/2+\varepsilon}.
\end{align*}
\end{lemma}

\subsection{Poisson summation}
\label{sec Poisson}

   We recall the following two dimensional Poisson summation formula, which follows from \cite[Lemma 2.7, Corollary 2.8]{G&Zhao4}.
\begin{lemma}
\label{Poissonsumformodd} Let $n \in \mathcal{O}_K$ be primary and $\leg {\cdot}{n}$ be the quadratic residue symbol modulo $n$. For any smooth function $W:\mr^{+} \rightarrow \mr$ of compact support,  we have for $X>0$,
\begin{align*}
   \sum_{\substack {m \in \mathcal{O}_K \\ (m,1+i)=1}}\leg {m}{n} W\left(\frac {N(m)}{X}\right)=\frac {X}{2N(n)}\leg {1+i}{n}\sum_{k \in
   \mathcal{O}_K}(-1)^{N(k)} g(k,n)\widetilde{W}\left(\sqrt{\frac {N(k)X}{2N(n)}}\right),
\end{align*}
   where
\begin{align}
\label{Wtdef}
   \widetilde{W}(t) =& \int\limits^{\infty}_{-\infty}\int\limits^{\infty}_{-\infty}W(N(x+yi))\widetilde{e}\left(- t(x+yi)\right)\dif x \dif y, \quad t \geq 0.
\end{align}
\end{lemma}

  When applying the above lemma in the proof of our result, we are led to consider the behaviors of a particular function. In the rest of this section, we include a result on this. We begin by defining
\begin{equation}
\label{Fdef}
  F_{y}(t) = \Psi(t) W_{\delta,\tau}\leg{\pi^2y}{2^5Xt},
\end{equation}
  where $\Psi(t)$ is a smooth function compacted in $[1,2]$ and $W_{\delta,\tau}$ is given in \eqref{Wdt}.

 The function that we are interested is then defined for $\xi>0$ and $\Re(w)>0$ by
\begin{equation}
\label{h}
h(\xi,w) = \int_0^{\infty} \widetilde{F}_{t}\left( \left(\frac{\xi}{t} \right )^{1/2} \right) t^{w-1} \,dt.
\end{equation}
  Thus, $h(\xi,w)$ is the Mellin transform of the function $\widetilde{F}_{t}\left( (\xi/t )^{1/2} \right)$. Here we recall that the Mellin transform $\widehat g(s)$ of a function $g$ is given by
\begin{align*}
 \widehat g(s) = \int_0^\infty g(t) t^{s-1} dt.
\end{align*}
  For further reference, we note that if we further assume that $g$ is support in $[1,2]$, then integration by parts implies that for $\Re(s)> 0$,  we have
\begin{align}
\label{gsest}
|\widehat g(s)| \ll \frac{2^{\Re(s)}}{|s|^n} g_{(n)},
\end{align}
  where we define for integers $n \ge 0$,
$$
g_{(n)} = \max_{0\le j\le n} \int_1^2 |g^{(j)}(t)|dt.
$$

  Now, we have the following result concerning analytical properties of $h(\xi,w)$.
\begin{lemma}\label{lem: properties of h(xi,w)}
Let $F_t$ be defined by \eqref{Fdef} and let $\xi >0$. The function $h(\xi,w)$ defined in \eqref{h} is an entire function of $w$ in $1 \geq \Re(w)>-1+|\Re(\tau)|$ such that for any $c$ with $1+ \Re(w)> c>\max(|\Re(\tau)|,\Re(w))$,
\begin{align}
\label{eq: integral in lemma for h(xi,w)}
\begin{split}
  h(\xi,w)
=&  \widehat{\Psi}(1+w) \xi^{w} \frac {\pi }{2\pi i}
\int\limits\limits_{(c)}\left(\frac{2^{5/2}}{\pi}\right)^{2s}
\Gamma_\delta(s) \left(\frac {X}{\xi} \right )^s \pi^{-2s+2w}\frac{\Gamma (s-w)}{\Gamma (1-s+w)} \frac{2s}{s^2 -\tau^2} ds.
\end{split}
\end{align}
 Moreover, when $1\geq \Re(w)> -1+|\Re(\tau)|$, we have
\begin{align}
\label{hbound}
h(\xi,w) \ll (1+|w|)^{3-2\Re(w)} \exp \Bigg( -\frac{1}{10}\frac{\xi^{1/4}}{X^{1/4}(|w|+1)^{1/2}}\Bigg) \xi^{\Re(w)} |\widehat{\Psi}(1+w)|.
\end{align}
\end{lemma}
\begin{proof}
  Notice that for any smooth function $W$, we can evaluate the function $\widetilde{W}(t)$ defined in \eqref{Wtdef} in polar coordinates as
\begin{align*}
     \widetilde{W}(t) =& 4\int\limits^{\pi/2}_0\int\limits^{\infty}_0\cos (2\pi t r\sin \theta)W(r^2) \ r \dif r \dif \theta.
\end{align*}
  Using this and the definition of $W_{\delta,\tau}(t)$ in \eqref{Wdt}, we deduce that, for $1+ \Re(w)> c>\max(|\Re(\tau)|,\Re(w))$,
\begin{align*}
& h(\xi,w) \\
=& 4\int\limits_0^{\infty} \int\limits^{\pi/2}_0 \int\limits^{\infty}_0\cos (2\pi \left(\frac{\xi}{t} \right )^{1/2} r \sin \theta)\Psi(r^2) W_{\delta,\tau}\leg{\pi^2t}{2^5Xr^2} \ r\dif r \dif \theta t^{w-1} \,dt  \\
=& 2\int\limits_0^{\infty} \int\limits^{\pi/2}_0 \int\limits^{\infty}_0\cos (r)\Psi((\frac {rt^{1/2}}{2\pi \xi^{1/2} \sin \theta})^2) W_{\delta,\tau}\left( \leg{\pi^2t}{2^5X}(\frac {rt^{1/2}}{2\pi \xi^{1/2} \sin \theta})^{-2} \right) \ \dif (\frac {rt^{1/2}}{2\pi \xi^{1/2} \sin \theta})^{2} \dif \theta t^{w-1} \,dt  \\
=& 2 \int\limits^{\pi/2}_0 \int\limits^{\infty}_0\cos (r) W_{\delta,\tau}\left( \leg{\pi^2}{2^5X}(\frac {r}{2\pi \xi^{1/2} \sin \theta})^{-2} \right) \ \dif (\frac {r}{2\pi \xi^{1/2} \sin \theta})^{2} \dif \theta \int\limits_0^{\infty} \Psi((\frac {rt^{1/2}}{2\pi \xi^{1/2} \sin \theta})^2) t^{w+1} \,\frac {dt}{t}  \\
=& 2 \widehat{\Psi}(1+w) \int\limits^{\pi/2}_0 \int\limits^{\infty}_0\cos (r) W_{\delta,\tau}\left( \leg{\pi^2}{2^5X}(\frac {r}{2\pi \xi^{1/2} \sin \theta})^{-2} \right)(\frac {r}{2\pi \xi^{1/2} \sin \theta})^{-2w-2} \ \dif (\frac {r}{2\pi \xi^{1/2} \sin \theta})^{2} \dif \theta   \\
=& 2 \widehat{\Psi}(1+w) \xi^{w} \int\limits^{\pi/2}_0 \int\limits^{\infty}_0\cos (r) \frac {1}{2\pi i}
\int\limits\limits_{(c_s)}\left(\frac{2^{5/2}}{\pi}\right)^{2s}
\Gamma_\delta(s)\left( \frac{1}{X}(\frac {r}{2\pi \xi^{1/2} \sin \theta})^{-2} \right)^{-s}(\frac {r}{2\pi \sin \theta})^{-2w-2} \ \frac{2s}{s^2 -\tau^2} ds \ \dif (\frac {r}{2\pi  \sin \theta})^{2} \dif \theta   \\
=& 4 \widehat{\Psi}(1+w) \xi^{w} \frac {1}{2\pi i}
\int\limits\limits_{(c_s)}\left(\frac{2^{5/2}}{\pi}\right)^{2s}
\Gamma_\delta(s) \left(\frac {X}{\xi} \right )^s (2\pi)^{-2s+2w} \int\limits^{\pi/2}_0 \int\limits^{\infty}_0\cos (r)r^{2s-2w-1}  \left(  \sin \theta \right)^{-2s+2w} \frac{2s}{s^2 -\tau^2} ds \ \dif r \dif \theta .
\end{align*}

 We apply the relation (see \cite[Section 2.4]{Gao1})
\begin{align*}
\int\limits^{\pi/2}_0 (\sin \theta )^{-u} \dif \theta
\int\limits^{\infty}_0\cos (r)r^{u}\frac {\dif r}{r}=\frac {\pi}{2}2^{u-1}\frac{\Gamma \leg{u}{2}}{\Gamma\leg{2-u}{2}},
\end{align*}
  to see that
\begin{align*}
\int\limits^{\pi/2}_0 (\sin \theta )^{-(2s-2w)} \dif \theta
\int\limits^{\infty}_0\cos (r)r^{2s-2w}\frac {\dif r}{r}=\frac {\pi}{2}2^{2s-2w-1}\frac{\Gamma (s-w)}{\Gamma (1-s+w)}.
\end{align*}

  Substituting this into the above expression for $h(\xi,w)$, we readily derive \eqref{eq: integral in lemma for h(xi,w)} by noticing that we can now take $c_s>\max(|\Re(\tau)|,\Re(w))$. This also implies that
$h(\xi,w)$ is an entire function of $w$ for $1 \geq \Re(w)>-1+|\Re(\tau)|$.

 It remains to establish \eqref{hbound}. For this, we set $c=\Re(s)$ and we may assume that $c \geq 2$ here. By apply Stirling's formula given in \cite[(5.112)]{iwakow}, we see that
\begin{align}
\label{gbound'}
\begin{split}
 & \left(\frac{2^{5/2}}{\pi}\right)^{2s}
\Gamma_\delta(s)  (\pi)^{-2s+2w}\frac{\Gamma (s-w)}{\Gamma (1-s+w)} \frac {2s}{s^2-\tau^2} \\
\ll & \left(\frac{2^{5/2}}{\pi^2}\right)^{2c}(1+|s|)^{2c-1}e^{-\pi|\Im(s)|}(1+|s-w|)^{2c-2\Re(w)-1} \\
 \ll &
(\frac {|s|}{e^{1/2}})^{2c-1}e^{-\pi|\Im(s)|}(1+|s|)^{2c-2\Re(w)-1} (1+|w|)^{2c-2\Re(w)-1}.
\end{split}
\end{align}

  This implies that
\begin{align*}
\begin{split}
& \int\limits\limits_{(c)}\left(\frac{2^{5/2}}{\pi}\right)^{2s}
\Gamma_\delta(s)(\pi)^{-2s+2w}\frac{\Gamma (s-w)}{\Gamma (1-s+w)} \frac {2s ds}{s^2-\tau^2}
\ll e^{-c}|c|^{4c-2\Re(w)-2}(1+|w|)^{2c-2\Re(w)-1}.
\end{split}
\end{align*}
  By taking
\begin{align*}
\begin{split}
 c=\max (2, \frac {\xi^{1/4}}{X^{1/4}(1+|w|)^{1/2}}),
\end{split}
\end{align*}
 we see that the bound given \eqref{hbound} follows.
\end{proof}

\subsection{Analytical behaviors of certain functions}
\label{sect: alybehv}

   Besides the function $h(\xi,w)$ considered in the previous section, we also need to know analytical behaviors of a few other functions that are needed in the paper.  We include several results in this section.
First, we note that the following result can be established similar to \cite[Lemma 5.3]{sound1}.
\begin{lemma}
\label{lem: nu-sum as an Euler product}
  Let $\alpha, l \in \mathcal O_K$ be primary and for each $k \in \mathcal O_K, k \neq 0$, we write $k$ uniquely by
\begin{equation}
\label{eq: defn of k1 and k2}
 k = k_1 k_2^2,
\end{equation}
 with $k_1$ square-free and $k_2 \in G$. For $\Re(s) >1+ |\Re (\delta)|$, we have
\begin{align*}
 & \sum_{\substack{n \equiv 1 \bmod {(1+i)^3} \\ (n,\alpha)=1}} \frac{r_{\delta}(n)}{N(n)^{s}}\cdot \frac {g(k,ln)}{N(n)^{1/2}} \ = \ L(s-\delta,\chi_{ik_1})L(s+\delta,\chi_{ik_1}) \prod_{\varpi \in G}\mathcal{G}_{\delta;\varpi}(s;k,l,\alpha) \\
 =: & \ L(s-\delta,\chi_{ik_1})L(s+\delta,\chi_{ik_1}) \mathcal{G}_{\delta}(s;k,l,\alpha).
\end{align*}
Here $\mathcal{G}_{\delta;\varpi}(s;k,l,\alpha)$ is defined by
\begin{align*}
\mathcal{G}_{\delta;\varpi}(s;k,l,\alpha) & =\begin{cases}
    \biggl( 1-\frac{1}{N(\varpi)^{s-\delta}}\leg{i k_1}{\varpi}
\biggr)\biggl( 1-\frac{1}{N(\varpi)^{s+\delta}}\leg{i k_1}{\varpi}
\biggr)  & \ \ \ \ \text{if }\varpi|2\alpha , \\
   \biggl( 1-\frac{1}{N(\varpi)^{s-\delta}}\leg{i k_1}{\varpi}
\biggr)\biggl( 1-\frac{1}{N(\varpi)^{s+\delta}}\leg{i k_1}{\varpi}
\biggr)\displaystyle\sum_{r=0}^{\infty} \frac{r_{\delta}(\varpi^r) }{N(\varpi)^{rs}}\frac{g(k, \varpi^{r+\text{\upshape{ord}}_\varpi(l)})}{N(\varpi)^{r/2}} & \ \ \ \ \text{if }\varpi \nmid 2\alpha.
\end{cases}
\end{align*}

Moreover, the function $\mathcal{G}_{\delta}(s;k,l,\alpha)$ is holomorphic for $\Re(s)>\frac{1}{2}+|\Re(\delta)|$ and satisfies the bound that uniformly for  $\Re(s) \ge \frac{1}{2} + |\Re(\delta)| + \epsilon$,
$$
| {\mathcal G}_{\delta}(s; k,l,\alpha)| \ll N(\alpha k)^{\epsilon}
N(l)^{\frac 12+\epsilon} N((l,k_2^2))^{\frac 12}.
$$
\end{lemma}

  Our next two lemmas provide bounds for certain dyadic sums involving $\mathcal{G}_{\delta}$ and $h(\xi,w)$.
\begin{lemma}\label{lem: version of Lemma 5.5 of Sound}
Let $K, L \geq 1$ be two integers and let $k_2$ be defined in \eqref{eq: defn of k1 and k2}. For any sequence of complex numbers $\delta_{l}$ satisfying $|\delta_{l}| \ll N(\ell)^{\varepsilon}$ and $\delta_l=0$ when $(l,2\alpha)\neq 1$, we have for $\Re(w)=-\frac{1}{2}+|\Re(\delta)| +\varepsilon$,
\begin{equation*}
\sum_{K\leq N(k)< 2K}\frac{1}{N(k_2)} \left| \sum_{\substack{N(l)=L}}^{2L-1} \frac{\delta_{l}}{\sqrt{N(l)}} \mathcal{G}_{\delta}(1+w;k,l,\alpha)\right|^2 \ll_{\varepsilon} (N(\alpha) LK)^{\varepsilon} L(L+K).
\end{equation*}
\end{lemma}
\begin{proof}
We write for any $k \neq 0, k \in \mathcal O_K$ as $k=u_k \displaystyle \prod_{\varpi \in G, \ a_i\geq 1} \varpi_i^{a_i}$ with $u_k \in U_K$. We define for the $\varpi_i$ appearing in this product,
\begin{equation}\label{eq: defn of a(k) and b(k)}
a(k)=\prod_i \varpi_i^{a_i+1} \ \ \ \ \text{and} \ \ \ \ b(k)=\prod_{a_i=1} \varpi_i \prod_{a_i\geq 2}\varpi_i^{a_i-1}.
\end{equation}
 It follows from the definition of $\mathcal{G}_{\delta}$ in Lemma~\ref{lem: nu-sum as an Euler product} and part (ii) of Lemma \ref{Gausssum} that we may write $l=gm$ with $g|a(k), g \in G$, $(m,k)=1$ and $m$ square-free, for otherwise we have $\mathcal{G}_{\delta}(1+w;k,l,\alpha)=0$. Further, we see that when $(l,2\alpha)=1$,
\begin{equation*}
\mathcal{G}_{\delta}(1+w;k,l,\alpha)=\sqrt{N(m)}\left( \frac{ik}{m}\right) \prod_{\substack{\varpi \in G \\ \varpi |m}} \left( 1+ \frac{r_{\delta}(\varpi)}{N(\varpi)^{1+w}}\left( \frac{i k_1}{\varpi}\right)\right)^{-1} \mathcal{G}_{\delta}(1+w;k,g,\alpha).
\end{equation*}
 Applying the above and the Cauchy-Schwarz inequality, we deduce that
\begin{equation*}
\sum_{K\leq N(k)< 2K}\frac{1}{N(k_2)} \left| \sum_{\substack{N(l)=L}}^{2L-1} \frac{\delta_{l}}{\sqrt{N(l)}} \mathcal{G}_{\delta}(1+w;k,l,\alpha)\right|^2 \ll_{\varepsilon} K^{\varepsilon} \sum_{K\leq N(k)< 2K}\frac{1}{N(k_2)}  \sum_{\substack{g|a(k) \\ N(g)<2L}} \Psi(k,g),
\end{equation*}
where
\begin{equation*}
\Psi(k,g) = \Bigg| \sum_{\substack{\frac{L}{N(g)}\leq  N(m) <\frac{2L}{N(g)} }} \frac{\mu^2_{[i]}(m) \delta_{gm}}{\sqrt{N(g)}}  \mathcal{G}_{\delta}(1+w;k,g,\alpha) \left( \frac{i k}{m}\right) \prod_{\substack{\varpi \in G \\ \varpi |m}} \left( 1+ \frac{r_{\delta}(\varpi)}{N(\varpi)^{1+w}}\left( \frac{i k_1}{\varpi}\right)\right)^{-1} \Bigg|^2.
\end{equation*}
  We use the bound for $\mathcal{G}_{\delta}$ given in Lemma~\ref{lem: nu-sum as an Euler product} in the above expression to see that
\begin{equation}\label{eq: Psi 1 after factoring out G0}
\Psi(k,g) \ll_{\varepsilon} (N(\alpha) K)^{\varepsilon}N(g)^{1+\varepsilon}\Bigg| \sum_{\substack{\frac{L}{N(g)}\leq N(m)<\frac{2L}{N(g)} }} \mu^2_{[i]}(m) \delta_{gm}  \left( \frac{i k}{m}\right) \prod_{\substack{\varpi \in G \\ \varpi |m}} \left( 1+ \frac{r_{\delta}(\varpi)}{N(\varpi)^{1+w}}\left( \frac{i k_1}{\varpi}\right)\right)^{-1} \Bigg|^2.
\end{equation}
 Note that as $\left( \frac{i k}{m}\right)\neq 0$,
\begin{align*}
\prod_{\substack{\varpi \in G \\ \varpi |m}}  \left( 1+ \frac{r_{\delta}(\varpi)}{N(\varpi)^{1+w}}\left( \frac{i k_1}{\varpi }\right)\right)^{-1}
& = \prod_{\substack{\varpi \in G \\ \varpi |m}}  \left( 1- \frac{r_{\delta}(\varpi)^2}{N(\varpi)^{2+2w}}\right)^{-1} \prod_{\substack{\varpi \in G \\ \varpi |m}}  \left( 1- \frac{r_{\delta}(\varpi)}{N(\varpi)^{1+w}}\left( \frac{i k_1}{\varpi}\right)\right) \\
& = \prod_{\substack{\varpi \in G \\ \varpi |m}}  \left( 1- \frac{r_{\delta}(\varpi)^2}{N(\varpi)^{2+2w}}\right)^{-1}\sum_{\substack{j|m \\ j \equiv 1 \bmod {(1+i)^3}}} \frac{\mu_{[i]}(j)r_{\delta}(j)}{N(j)^{1+w}} \left( \frac{i k_1}{j}\right).
\end{align*}
 Using this in \eqref{eq: Psi 1 after factoring out G0}, we obtain via another application of the Cauchy-Schwarz inequality that
\begin{equation*}
\Psi(k,g) \ll_{\varepsilon} (N(\alpha) K)^{\varepsilon}N(g)^{1+\varepsilon} \sum_{N(j)<\frac{2L}{N(g)} } \Bigg| \sum_{\substack{\frac{L}{N(g)}\leq N(m)<\frac{2L}{N(g)} \\ j|m }} \mu^2_{[i]}(m) \delta_{gm} \left( \frac{i k}{m}\right) \prod_{\substack{\varpi \in G \\ \varpi |m}}  \left( 1- \frac{r_{\delta}(\varpi)^2}{N(\varpi)^{2+2w}} \right)^{-1}\Bigg|^2.
\end{equation*}

  We relabel $m$ by $jm$ and note that for all $\varpi |m$ and $\Re(w)=-\frac{1}{2}+|\Re(\delta)| +\varepsilon$, we have
\begin{align*}
 \mu^2_{[i]}(j)\left( \frac{i k}{j}\right) \prod_{\substack{\varpi \in G \\ \varpi |j }} \left( 1- \frac{r_{\delta}(\varpi)^2}{N(\varpi)^{2+2w}} \right)^{-1} \ll N(j)^{\varepsilon}.
\end{align*}
  This implies that
\begin{equation}\label{eq: after relabeling m as jm in Psi 1}
\Psi(k,g) \ll_{\varepsilon} (N(\alpha) LK)^{\varepsilon}N(g)^{1+\varepsilon} \sum_{N(j)<\frac{2L}{N(g)} } \Bigg| \sum_{\substack{\frac{L}{N(gj)}\leq N(m)<\frac{2L}{N(gj)} \\ (m,2\alpha j)=1}} \mu^2_{[i]}(m) \delta_{gjm} \left( \frac{i k}{m}\right) \prod_{\substack{\varpi \in G \\ \varpi |m}} \left( 1- \frac{r_{\delta}(\varpi)^2}{N(\varpi)^{2+2w}} \right)^{-1}\Bigg|^2.
\end{equation}

 Notice that $g|a(k)$ implies $b(g)|k$ by \eqref{eq: defn of a(k) and b(k)}. We may thus relabel such $k$ by $f b(g)$ to deduce from \eqref{eq: after relabeling m as jm in Psi 1} that
\begin{align}
\label{eq: after relabeling k as fb(g)}
\begin{split}
 & \sum_{K\leq N(k)< 2K}\frac{1}{N(k_2)}  \sum_{\substack{g|a(k) \\ N(g)<2L}} \Psi(k,g) \\
\leq & \sum_{N(g)<2L} \sum_{\substack{K\leq N(k)< 2K \\ b(g)|k} }\frac{1}{N(k_2)}  \Psi(k,g)
= \sum_{N(g)<2L}\sum_{\frac{K}{N(b(g))}\leq N(f)< \frac{2K}{N(b(g))}} \frac{1}{N(k_2)}  \Psi(f b(g),g) \\
\ll_{\varepsilon} &  (N(\alpha)  L K)^{\varepsilon}\sum_{N(g)<2L}N(g)^{1+\varepsilon} \sum_{\frac{K}{N(b(g))}\leq N(f)< \frac{2K}{N(b(g))}} \frac{1}{N(k_2)} \\
& \times \sum_{N(j)<\frac{2L}{N(g)} } \Bigg| \sum_{\substack{\frac{L}{N(gj)}\leq N(m)<\frac{2L}{N(gj)}}} \mu^2_{[i]}(m) \delta_{gjm} \left( \frac{i f b(g)}{m}\right) \prod_{\substack{\varpi \in G \\ \varpi |m}} \left( 1- \frac{r_{\delta}(\varpi)^2}{N(\varpi)^{2+2w}} \right)^{-1}\Bigg|^2 .
\end{split}
\end{align}

  On writing $f=f_1f_2^2$ with $f_1$ square-free and $f_2 \in G$, we observe that the relation $f b(g)=k_1k_2^2$ implies that $f_2| k_2$, so that $N(k_2)^{-1}\ll N(f_2)^{-1}$. Applying this in \eqref{eq: after relabeling k as fb(g)}, we see that the assertion of our lemma follows from Lemma~\ref{lem: estimates for character sums}.
 \end{proof}

\begin{lemma}\label{lem: version of Lemma 5.4 of Sound}
Let $K, L \geq 1$ be two integers and let $N(\alpha)\leq X, X>0$. For $\Re(w)=-\frac{1}{2}+|\Re(\delta)|+\varepsilon$ and any sequence of complex numbers $\gamma_{l}$ satisfying $|\gamma_{l}|\leq 1$, the expression
\begin{align}
\label{Gklsum}
\sum_{\substack{ K\leq N(k)<2K  }}\frac{1}{N(k_2)} \Bigg| \sum_{\substack{N(l)=L \\ (l,2\alpha)=1}}^{2L-1} \frac{\gamma_{l}}{N(l)}  \mathcal{G}_{\delta}(1+w;k,l,\alpha) h(\frac{N(k)X}{2N(\alpha^2l)}, w) \Bigg|^2
\end{align}
is bounded by
\begin{align*}
\ll_{\varepsilon} |\widehat{\Psi}(1+w)|^2(1+|w|)^{8-4|\Re(\delta)|+\varepsilon}\frac{N(\alpha)^{2-4|\Re(\delta)|+\varepsilon}L^{2-2|\Re(\delta)|+\varepsilon}
K^{2|\Re(\delta)|+\varepsilon} }{X^{1-2|\Re(\delta)|-\varepsilon}} \exp\Bigg( -\frac{1}{20}\frac{\sqrt[4]{K}}{\sqrt[4]{N(\alpha)^2 L(1+|w|^2)}}\Bigg),
\end{align*}
and also by
\begin{align*}
\ll_{\varepsilon} ((1+|w|)N(\alpha) LKX)^{\varepsilon}|\widehat{\Psi}(1+w)|^2  \Big(\frac{N(\alpha)^2 L(1+|w|^2)}{K}\Big)^{2|\Re(\tau)|-2|\Re(
\delta)|}
\frac{N(\alpha)^{2}L}
{K X^{1-2|\Re(\delta)|}}(K+L).
\end{align*}
\end{lemma}
\begin{proof}
We apply Lemma~\ref{lem: properties of h(xi,w)}, Lemma~\ref{lem: nu-sum as an Euler product} to bound respectively $h(\xi,w)$ and $\mathcal{G}_{\delta}$ to see that the expression in \eqref{Gklsum} is
\begin{align*}
\ll & |\widehat{\Psi}(1+w)|^2(1+|w|)^{8-4|\Re(\delta)|+\varepsilon}\frac{N(\alpha)^{2-4|\Re(\delta)|+\varepsilon}L^{-2|\Re(\delta)|+\varepsilon}
K^{2|\Re(\delta)|+\varepsilon} }{X^{1-2|\Re(\delta)|-\varepsilon}} \exp\Bigg( -\frac{1}{20}\frac{\sqrt[4]{K}}{\sqrt[4]{N(\alpha)^2 L(1+|w|^2)}}\Bigg) \\
& \times \sum_{\substack{ K\leq N(k)<2K }}\frac{1}{N(k)N(k_2)} \Bigg(\sum_{\substack{N(l)=L \\ (l,2\alpha)=1}}^{2L-1} N((l,k_2^2))^{\frac{1}{2}} \Bigg)^2 \\
\ll & |\widehat{\Psi}(1+w)|^2(1+|w|)^{8-4|\Re(\delta)|+\varepsilon}\frac{N(\alpha)^{2-4|\Re(\delta)|+\varepsilon}L^{-2|\Re(\delta)|+\varepsilon}
K^{2|\Re(\delta)|+\varepsilon} }{X^{1-2|\Re(\delta)|-\varepsilon}} \exp\Bigg( -\frac{1}{20}\frac{\sqrt[4]{K}}{\sqrt[4]{N(\alpha)^2 L(1+|w|^2)}}\Bigg) \\
& \times \sum_{\substack{ K\leq N(k)<2K }}\frac{1}{N(k)N(k_2)} \Bigg(\sum_{\substack{N(l)=L \\ (l,2\alpha)=1}}^{2L-1} N(k_2) \Bigg)^2.
\end{align*}
 Upon writing $N(k)=N(k_1k^2_2)$,  we readily deduce the first bound of the lemma from the above estimation.

 To derive the second bound, we set $c=|\Re(\tau)|+\varepsilon$ to recast the integral in \eqref{eq: integral in lemma for h(xi,w)} as
\begin{align*}
\frac{1}{2\pi i} \int\limits_{(|\Re(\tau)|+\varepsilon)} g(s,w)\left(\frac{X}{\xi}\right)^s\,ds.
\end{align*}
 This implies that
\begin{align*}
& \Bigg| \sum_{\substack{N(l)=L \\ (l,2\alpha )=1}}^{2L-1} \frac{\gamma_{l}}{N(l)}  \mathcal{G}_{\delta}(1+w;k,l,\alpha) h(\frac{N(k)X}{2N(\alpha^2l)}, w) \Bigg| \\
 \ll & |\widehat{\Psi}(1+w)|\left(\frac{N(\alpha)^{1+2|\Re(\tau)|-2|\Re(\delta)|+\varepsilon}  }{K^{\frac{1}{2}+|\Re(\tau)|-|\Re(\delta)|-\varepsilon}X^{\frac{1}{2}-|\Re(\delta)|-\varepsilon}}\right)\int\limits_{(|\Re(\tau)|+\varepsilon)}\Bigg| g(s,w)\sum_{\substack{N(l)=L \\ (l,2\alpha)=1}}^{2L-1}  \frac{\gamma_{l}}{N(l)^{1+w-s}} \mathcal{G}_{\delta}(1+w;k,l,\alpha)   \Bigg|\,|ds|.
\end{align*}
 Notice that \eqref{gbound'} is still valid with $c=|\Re(\tau)|+\varepsilon$ and $\Re(w)=-\frac{1}{2}+|\Re(\delta)|+\varepsilon$ so that it implies that $g(s,w)\ll_{\varepsilon}(1+|w|)^{2|\Re(\tau)|-2|\Re(\delta)|+\varepsilon}\exp(-\frac{\pi}{2} |\Im(s)|)$. We apply this estimation and the Cauchy-Schwarz inequality to deduce that
\begin{align*}
\begin{split}
& \Bigg| \sum_{\substack{N(l)=L \\ (l,2\alpha )=1}}^{2L-1}  \frac{\gamma_{l}}{N(l)}  \mathcal{G}_{\delta}(1+w;k,l,\alpha) h(\frac{N(k)X}{2N(\alpha^2l)}, w)\Bigg| \\
\ll &  (1+|w|)^{2|\Re(\tau)|-2|\Re(\delta)|+\varepsilon}|\widehat{\Psi}(1+w)| \left(\frac{N(\alpha)^{1+2|\Re(\tau)|-2|\Re(\delta)|+\varepsilon} }{K^{\frac{1}{2}+|\Re(\tau)|-|\Re(\delta)|-\varepsilon}X^{\frac{1}{2}-|\Re(\delta)|-\varepsilon}}\right) \\
& \times  \int\limits_{(|\Re(\tau)|+\varepsilon)}\Bigg (\exp(-\tfrac{\pi}{2}|\Im(s)|)\Bigg| \sum_{\substack{N(l)=L \\ (l,2\alpha )=1}}^{2L-1}  \frac{\gamma_{l}}{N(l)^{1+w-s}} \mathcal{G}_{\delta}(1+w;k,l,\alpha)   \Bigg|^2\,|ds| \Bigg )^{1/2}.
\end{split}
\end{align*}
 By inserting the above bound into \eqref{Gklsum} and applying Lemma~\ref{lem: version of Lemma 5.5 of Sound}, we obtain the second bound of the lemma.
\end{proof}

   In the remaining of the section, we include two more results concerning various functions studied in this paper.
\begin{lemma}
\label{Lemma 3.16}  Let $R$ be a polynomial satisfying $R(0)=R^{\prime}(0)=0$.
Let $g$ be a multiplicative function satisfying $g(\varpi)=1+O(N(\varpi)^{-\nu})$
for some fixed $\nu >0$.  Let $u$ and $v$ be two
bounded complex numbers such that $\Re(u+v)$ and $\Re
(u-v)$ are $\ge - D/\log y$ for an absolute positive
constant $D$ and a large real
number $y$. Then we have for $\Re(s)>1 + D/\log y$,
\begin{align}
\label{gEulerprod}
\sum_{\substack{ n \equiv 1 \bmod (1+i)^3 }}
\frac{r_{v}(n) \mu_{[i]}(nc)}{N(n)^{s+u}} g(n)
= \frac{\mu_{[i]}(c) G(s,c;u,v)}{\zeta_K(s+u+v)\zeta_K(s+u-v)},
\end{align}
where $G(s,c;u,v)=\prod_{\substack{ \varpi \in G }} G_{\varpi} (s,c;u,v)$ is a holomorphic function in $\Re(s)> \max(\frac 12,1-\nu)
+ D/\log y$ defined by
$$
G_{\varpi}(s,c;u,v):=
\begin{cases}
(1-\frac 1{N(\varpi)^{s+u+v}})^{-1} (1-\frac{1}{N(\varpi)^{s+u-v}})^{-1}
&\text{if } \varpi|2c, \\
(1-\frac 1{N(\varpi)^{s+u+v}})^{-1} (1-\frac{1}{N(\varpi)^{s+u-v}})^{-1} (1-\frac{g(\varpi)r_v(\varpi)}
{N(\varpi)^{s+u}})
&\text{otherwise}.
\end{cases}
$$

Moreover, for any odd $c \in \mathcal O_K$ with $N(c) \le y$,  we have
\begin{align}
\label{sumoverg}
\begin{split}
\sum_{\substack{ N(n) \le y/N(c) \\ n \equiv 1 \bmod (1+i)^3 }}
\frac{r_{v}(n)\mu_{[i]}(nc)}{N(n)^{1+u}}&g(n) R\leg{\log (y/N(cn))}{\log y}
=  O\Big(\frac{E(c)}{\log^2 y} \leg{y}{N(c)}^{-\Re(u) +|\Re(v)|}
\exp(-A_0\sqrt{\log (y/N(c))})\Big)\\
&+ \mathop{\text{Res}}_{s=0}
\frac{\mu_{[i]}(c)G(s+1,c;u,v)}{s\zeta_K(1+s+u+v)\zeta_K(1+s+u-v)
} \sum_{k=0}^{\infty} \frac{1}{(s\log y)^k} R^{(k)}
\Big(\frac{\log (y/N(c))}{\log y}\Big),
\end{split}
\end{align}
where $A_0>0$ is an absolute constant and $E(c)=\prod_{\substack{ \varpi \in G \\ \varpi |c}}
(1+1/\sqrt{N(\varpi)})$.
\end{lemma}
\begin{proof} First note that we can establish \eqref{gEulerprod} by considering
Euler products. To prove \eqref{sumoverg},  we may assume that $N(c) \le y/2$.
We then apply the Taylor expansion $R(x)=\sum_{j=0}^{\infty}
\frac{R^{(j)}(0)}{j!} x^j = \sum_{j=2}^{\infty} \frac{R^{(j)}(0)}{j!} x^j$
to write the sum in \eqref{sumoverg} as
\begin{align*}
&\sum_{j=2}^{\infty} \frac{R^{(j)}(0)}{(\log y)^{j}} \frac{1}{j!} \sum_{\substack{ N(n) \le y/N(c) \\ n \equiv 1 \bmod (1+i)^3 }} \frac{r_{v}(n)\mu_{[i]}(nc)}{N(n)^{1+u}} g(n)
\log^{j} \leg{y}{N(cn)}.
\end{align*}
   We note that the inner sum above can be regarded as a Riesz type means so that we can apply the treatment given in \cite[Sect 5.1]{MVa1} to further write the above sums as
\begin{align}
\label{integralformLemma3.16}
& \sum_{j=2}^{\infty} \frac{R^{(j)}(0)}{(\log y)^j}
\frac{1}{2\pi i} \int\limits_{(\frac{D+1}{\log (y/N(c))})}
\frac{\mu_{[i]}(c) G(s+1,c;u,v)}{\zeta_K(1+s+u+v)\zeta_K(1+s+u-v)} \Big(\frac y{N(c)}\Big)^s
\frac{ds}{s^{j+1}}.
\end{align}

  To evaluate the integral above, we set $T=\exp(\sqrt{\log (y/N(c))})$ and apply the zero free region for $\zeta_K(s)$ (see \cite[Section 8.4]{MVa1}) to choose a positive constant $A_1$ such that $\zeta_K(1+s+u+v)\zeta_K(1+s+u-v)$ has no zeros in the region $\Re(s) \geq -\Re(u) +|\Re(v)|-A_1/\log T, \Im(s) \leq T$. We then notice that we have $G(s+1,c;u,v)\ll E(c)$ uniformly for $\Re(s) \geq -\Re(u) +|\Re(v)|-A_1/\log T$. Moreover,  similar to the bound given for the Riemann zeta function in \cite[Theorem 6.7]{MVa1}, we have the following estimation for $1/\zeta_K(1+s)$ which asserts that for $\Re(s) > -A_1/\log T$,
\begin{align}
\label{zetarecibound}
\frac{1}{\zeta_K(1+s)} \ll \min (1, \log^2 \Im(s)) .
\end{align}

  We now truncate the integral in \eqref{integralformLemma3.16} to the line segment
$\frac{D+1}{\log (y/N(c))} -iT$ to $\frac{D+1}{\log (y/N(c))} +iT$ with
$T=\exp(\sqrt{\log (y/N(c))})$. By applying the above estimations for $G(s+1,c;u,v)$ and $1/\zeta_K(1+s)$, we see that the error introduced by doing so is
\begin{align}
\label{errtruncate}
\ll E(c)(\log y/N(c))^2/T^2.
\end{align}

 We then shift the remaining integral to the line segment $-\Re(u) +|\Re(v)|
-A_1/\log T$. Again by the above estimations for $G(s+1,c;u,v)$ and $1/\zeta_K(1+s)$, we see that the two horizontal integrals are
\begin{align}
\label{errhzt}
 \ll E(c)(\log (y/N(c)))^2\frac 1{T^3}.
\end{align}

   For the vertical integral, we note that as $\min (\Re(u+v), \Re
(u-v) ) \geq - D/\log y$, we have
\begin{align*}
 -\Re(u) +|\Re(v)| - A_1/\log T \ll -1/\sqrt{\log (y/N(c))}.
\end{align*}
  We can thus divide the vertical integral into two parts, one over the segment $|\Im(s)| \leq 1/\sqrt{\log (y/N(c))}$ and one over the rest. We apply \eqref{zetarecibound} and the bound $G(s+1,c;u,v)\ll E(c)$  to see that the vertical integral is
\begin{align*}
 \ll & E(c)((\log (y/N(c)))^{j/2}+(\log (y/N(c)))^2(\log (y/N(c)))^{j/2})(y/N(c))^{-\Re(u) +|\Re(v)| - A_1/\log T}.
\end{align*}
 As $R$ is a polynomial, we know that $R^{(j)}(0) \neq 0$ only for finitely many $j$ so that we may assume that $j$ is bounded. It follows that for an appropriate positive constant $A_0$, the above is
\begin{align}
\label{errvertint}
\ll E(c) \leg{y}{N(c)}^{-\Re(u) +|\Re(v)|}
\exp(-A_0\sqrt{\log (y/N(c))}).
\end{align}

   Furthermore, we note that we encounter a multiple
pole at $s=0$ in the above process. Thus, by combining \eqref{errtruncate}, \eqref{errhzt} and \eqref{errvertint}, we conclude that
the expression given in \eqref{integralformLemma3.16} is
\begin{align}
\label{simplifiedintegralformLemma3.16}
\begin{split}
=& \mathop{\text Res}_{s=0}
\frac{\mu_{[i]}(c) G(s+1,c;u,v)}{s\zeta_K(1+s+u+v)\zeta_K(1+s+u-v)}
\sum_{j=2}^{\infty}\frac{ R^{(j)}(0) (y/N(c))^s}{s^{j} (\log y)^j}
\\
&
+ O\Big( \frac{E(c)}{(\log y)^2} \leg{y}{N(c)}^{-\Re(u) +|\Re(v)|}
\exp(-A_0\sqrt{\log (y/N(c))})\Big).
\end{split}
\end{align}

 We now calculate the residue above using the Taylor expansion of $(y/N(c))^s$ around $s=0$ and discarding the powers of $s$ that are $\geq j+1$  since they make no contributions. This way, we see that the residue equals
$$
\sum_{j=2}^{\infty}\frac{ R^{(j)}(0)}{s^{j} (\log y)^j}
\Big(\sum_{l\le j} \frac{s^l}{l!} (\log (y/N(c))^l\Big)
= \sum_{k=0}^{\infty}
\frac{s^{-k}}{(\log y)^k}  \sum_{l=0}^{\infty}
\frac{R^{(k+l)}(0)}{l!} \Big(\frac{\log (y/N(c))}{\log y}\Big)^l=\sum_{k=0}^{\infty}
\frac{s^{-k}}{(\log y)^k} R^{(k)}\Big(\frac{\log (y/N(c))}{\log y}\Big),
$$
where the first equality above following by setting $k=j-l$ while noting that
$R(0)=R'(0)=0$.  Applying this in \eqref{simplifiedintegralformLemma3.16} allows us to deduce \eqref{sumoverg} and this completes the proof of the lemma.
\end{proof}

  For our next result, we define for odd primary primes $\varpi$,
\begin{align*}
h_{w}(\varpi) =& \Big( 1+\frac 1{N(\varpi)} +\frac 1{N(\varpi)^{1+2w}} -
\frac{N(\varpi)^{-2\delta}+N(\varpi)^{2\delta}}{N(\varpi)^{2+2w}} +\frac{1}{N(\varpi)^{3+4w}}\Big),  \\
H_{w}(\varpi) =& 1+ \frac{1}{N(\varpi)^{1+2w}} -\frac{r_{\delta}(\varpi)^2}{
N(\varpi)^{1+2w} h_{w}(\varpi)},
\end{align*}
and extend the above definitions multiplicatively to functions $h_{w}(n), H_{w}(n)$ on primary odd, square-free algebraic numbers $n \in \mathcal O_K$.

  Moreover, we write any odd $l \in \mathcal O_K$ as $l=l_1l_2^2$, with $l_1$ being square-free and $l_2 \in G $.
We define an absolutely convergent function $\eta_w(s;l)= \prod_{\substack{\varpi \in G}} \eta_{\varpi;w}(s;l)$ for complex numbers $w, s$ in the region $|\Re(w)|\le \frac 14$ and $\Re(s) >\frac 12$ such that
$\eta_{1+i;w}(s;l)= (1-2^{-s-2w})(1-2^{-s})(1-w^{-s+2w})$ and for primes
$(\varpi, 2)=1$,
\begin{align}
\label{etadef}
\begin{split}
\eta_{\varpi;w}(s;l) =
\begin{cases}
\leg{N(\varpi)}{N(\varpi)+1} \Big(1-\frac{1}{N(\varpi)^s}\Big) \Big(1+\frac 1{N(\varpi)} +\frac 1{N(\varpi)^s}
- \frac{N(\varpi)^{2w}+N(\varpi)^{-2w}}{N(\varpi)^{s+1}} +\frac{1}{N(\varpi)^{2s+1}}\Big)
&\text{if } \varpi \nmid l,\\
\leg{N(\varpi)}{N(\varpi)+1} \Big(1-\frac{1}{N(\varpi)^s}\Big) &\text{if } \varpi|l_1,\\
\leg{N(\varpi)}{N(\varpi)+1} \Big(1-\frac{1}{N(\varpi)^{2s}}\Big) &\text{otherwise}.
\end{cases}
\end{split}
\end{align}

  Lastly, we denote ${\mathcal C}$ for a closed contour (oriented counter-clockwise) containing the points $\pm \tau$ with perimeter
length being $\ll |\delta_1|$. Also,  for $w \in {\mathcal C}$
we have $|\Re(w)| \le |\tau| +C/\log X$, $|\Im(w)| \le
C/\log X$ for
some absolute constant $C$ and such that $\min (
|w^2-\tau^2|, |w^2-\delta^2|) \ge
\epsilon^2/(3\log^2 X)$ . We then have the following result.
\begin{lemma}
\label{Lemma 3.17}  With notations as above, we have for $w$ on the contour ${\mathcal C}$ and $x\ge 2$,
\begin{align}
\label{sumovergammaofg}
\begin{split}
& \eta_\delta(1+2w;1)  \sum_{\substack{ N(\gamma) \le x\\ \gamma \equiv 1 \bmod (1+i)^3 }}
\frac{\mu^2_{[i]} (\gamma) H_w(\gamma)}{N(\gamma)^{1+2\tau} h_w(\gamma)}
G(1,\gamma;\delta_1+w,\delta)G(1,\gamma;\delta_2+w,\delta)
\\
=&  \zeta_K(1+2\tau)(1-x^{-2\tau})(1+O(|w-\tau|)) + O(x^{-2\tau}).
\end{split}
\end{align}

Also, for any smooth function
$R$ on $[0,1]$ and $1\le y\le x$, we have
\begin{align}
\label{sumovergammaofgsmooth}
\begin{split}
&\eta_\delta(1+2w;1) \sum_{\substack{ y \le N(\gamma) \le x\\ \gamma \equiv 1 \bmod (1+i)^3 }}
\frac{\mu^2_{[i]} (\gamma) H_w(\gamma)}{N(\gamma)^{1+2\tau} h_w(\gamma)}
G(1,\gamma;\delta_1+w,\delta)G(1,\gamma;\delta_2+w,\delta) R\Big(
\frac{\log N(\gamma)}{\log x}\Big)
\\
=&  \frac {\pi}{4} \cdot (1+O(|\delta_1|)) \int_y^x R\Big(\frac{\log t}{\log x}\Big) \frac{dt}
{t^{1+2\tau}}.
\end{split}
\end{align}
\end{lemma}
\begin{proof} Applying the definition of $G(s,\gamma;u,v)$ given in
Lemma \ref{Lemma 3.16} allows us to write the expression given in \eqref{sumovergammaofg} as
$$
\eta_\delta(1+2w;1)G(1,1;\delta_1+w,\delta) G(1,1;\delta_2+w,\delta)
 \sum_{\substack{ N(\gamma) \le x \\ \gamma \equiv 1 \bmod (1+i)^3
}} \frac{f_w(\gamma)}{N(\gamma)^{1+2\tau}},
$$
 where
$$
f_w(\gamma)= \mu^2_{[i]}(\gamma)\frac{H_w(\gamma)}{h_w(\gamma)}
\prod_{\substack{ \varpi \in G \\ \varpi|\gamma}} \Big( 1 -\frac{r_\delta(\varpi)}{N(\varpi)^{1+\delta_1+w} h_w(\varpi)}
\Big)^{-1} \Big( 1 -\frac{r_\delta(\varpi)}{N(\varpi)^{1+\delta_2+w} h_w(\varpi)}
\Big)^{-1}.
$$
 Note here that $f_w(\gamma)$ is a multiplicative function such that $f_w(\varpi)=1+O(1/\sqrt{N(\varpi)})$.  This allows us to write $\sum_{\gamma \equiv 1 \bmod (1+i)^3} f_w(\gamma)/N(\gamma)^s
=\zeta_K(s) F_w(s)$ so that $F$ is holomorphic in $\Re(s) >\frac 12$.

 Now we apply Perron's formula as given in \cite[Theorem 5.2, Corollary 5.3]{MVa1} to obtain that for $\sigma=-2\tau+\varepsilon$, $T=x^{1+2\tau+\varepsilon}$,
\begin{align*}
 \sum_{\substack{ N(\gamma) \le x\\ \gamma \equiv 1 \bmod (1+i)^3 }} \frac{f_w(\gamma)}{N(\gamma)^{1+2\tau}} &=
\frac 1{2 \pi i}\int^{\sigma+i T}_{\sigma-i T} \zeta_K(1+2\tau+s)F_w(1+2\tau+s)\frac {x^s}{s}ds +R,
\end{align*}
  where
\begin{align*}
  R \ll \sum_{\substack{ x/2<N(n)<x \\ N(n) \neq x}} \frac{f_w(n)}{N(n)^{1+2\tau}} \min (1, \frac {x}{T|x-N(n)|})+\frac {4^{\sigma}+x^{\sigma}}{T}\sum_{n \equiv 1 \bmod (1+i)^3}\frac {\frac{|f_w(n)|}{N(n)^{1+2\tau}}}{N(n)^{\sigma}}.
\end{align*}
  It is then easy to see that
\begin{align}
\label{Rbound1}
  R \ll \sum_{\substack{ 1 \leq k<x}}\frac {x}{T k}+\frac {x^{-2\tau+\varepsilon}}{T}\ll \frac {x \log x}{T} \ll x^{-2\tau}.
\end{align}
 We now shift the contour of integration to $\sigma_1=-2\tau-1/2-\varepsilon$ to see that we have
\begin{align}
\label{hbound1}
 \frac 1{2 \pi i}\int^{\sigma_1 \pm i T}_{\sigma \pm i T} \zeta_K(1+2\tau+s)F_w(1+2\tau+s)\frac {x^s}{s}ds \ll x^{-2\tau}.
\end{align}
  Moreover, by applying the following subconvexity bound for $\zeta_k(s)$ on the critical line given in \cite{HB1988},
\begin{align*}
 \zeta_K(s) \ll (1+|s|)^{1/3+\varepsilon}, \quad \Re(s)=1/2,
\end{align*}
  we deduce that
\begin{align}
\label{vbound}
 \frac 1{2 \pi i}\int^{\sigma_1+i T}_{\sigma_1-i T} \zeta_K(1+2\tau+s)F_w(1+2\tau+s)\frac {x^s}{s}ds \ll \int^{T}_{-T} (1+|t|)^{1/3+\varepsilon}\frac {x^{-2\tau-1/2+\varepsilon}}{1+|t|}|dt| \ll x^{-2\tau}.
\end{align}
  By combining \eqref{Rbound1}, \eqref{hbound1} and \eqref{vbound}, together with the observation that $\zeta_K(1+2\tau) = \frac {\pi}{4} \cdot 1/(2\tau) +O(1)$, $F_w(1+2\tau) =
F_w(1)+ O(\tau)$, we then conclude that
\begin{align*}
 \sum_{\substack{ N(\gamma) \le x\\ \gamma \equiv 1 \bmod (1+i)^3 }} \frac{f_w(\gamma)}{N(\gamma)^{1+2\tau}} &=
 \zeta_K(1+2\tau)(1-x^{-2\tau})F_w(1+2\tau)
+ O(x^{-2\tau}).
\end{align*}
  If we set for brevity that $F(w)=\eta_\delta(1+2w;1) G(1,1;\delta_1+w,\delta)
G(1,1;\delta_2+w,\delta)F_w(1+2\tau)$, then a little
calculation implies that $F(\tau)=1$ and $F(w) = F(\tau) + O(|w-\tau|)$. This implies \eqref{sumovergammaofg}, which in turn implies \eqref{sumovergammaofgsmooth} by partial summation and this completes the proof of
the lemma.
\end{proof}

\section{Plan of the proof}
\label{sec planofproof}

We define for any large number $X$ and any smooth function $\Phi$ supported in $[1,2]$,
\begin{align*}
  {\mathcal N}(X,\Phi)=\sum_{\substack{(d,2)=1 \\ L(\beta,\chi_{(1+i)^5d})=0, 0< \beta  \leq 1 }}
\mu^2_{[i]}(d)  \Phi(\frac {N(d)}{X}),
\end{align*}
  so that $ {\mathcal N}(X,\Phi)$ is a weighted count on the number of odd,  square-free algebraic integers $d \in \mathcal O_K$ such that $X \leq N(d) \leq 2X$ and that $L(s,\chi_{(1+i)^5d})$ has a non-trivial real zero. The proof of Theorem \ref{thm: nonvanishing} is based on an estimation of ${\mathcal N}(X,\Phi)$. To achieve this, we apply the following argument principle given in \cite[Lemma 2.1]{C&S}, which is originally due to A. Selberg \cite{Selberg46}.
\begin{lemma}
\label{Lemma 2.1}
 Let $f(z)$ be a holomorphic function that does not vanish in the region $\Re(z) \ge W$.
Let ${\mathcal B}$ be the rectangular box with vertices $W_0\pm iH$,
$W_1 \pm iH$ where $H>0$ and $W_0, W_1$ are two real numbers such that $W_0 < W < W_1$.  Then
\begin{align*}
\begin{split}
& 4H
\sum_{\substack{\beta +i\gamma \in {\mathcal B}\\ f(\beta+i\gamma) =0 }}
\cos \Big ( \frac{\pi\gamma}{2H} \Big ) \sinh \Big ( \frac{\pi(\beta-W_0)}{2H} \Big ) \\
=& \int_{-H}^{H} \cos \Big ( \frac{\pi t}{2H}\Big ) \log |f(W_0+it)| dt+\int_{W_0}^{W_1}\sinh \Big ( \frac{\pi (\alpha-W_0)}{2H} \Big )
\log |f(\alpha+iH)f(\alpha-iH)| d\alpha
\\
&-\Re \int_{-H}^{H} \cos \Big(\pi \frac{W_1-W_0 +it}{2iH}\Big)
\log f(W_1+it) dt.
\end{split}
\end{align*}
\end{lemma}

  To proceed further, we define for any sequence $\{a_n\}_{n \in \mathcal O_K}$ of complex numbers and any smooth function $\Phi$ supported in $[1,2]$,
$$
{\mathcal S}(a_d;\Phi) = {\mathcal S}(a_d;\Phi,X) = \frac{1}{X} \sum_{\substack{(d,2)=1 } }
\mu^2_{[i]}(d) a_d \Phi(\frac {N(d)}{X}).
$$
  For a real parameter $1< Y \leq \sqrt{2X}$ to be determined later, we write
\begin{equation*}
\mu_{[i]}^2(n) = M_Y(n) + R_Y(n),
\end{equation*}
with
\begin{equation*}
M_Y(n) = \sum_{\substack{\ell \in G \\ \ell^2 \mid n \\ N(\ell) \leq Y}} \mu_{[i]}(\ell), \ \ \ \ \ R_Y(n) = \sum_{\substack{\ell \in G \\ \ell^2 \mid n \\ N(\ell) > Y}} \mu_{[i]}(\ell).
\end{equation*}
  We then deduce that ${\mathcal S}(a_d;\Phi)={\mathcal S}_M(a_d;\Phi) + O({\mathcal S}_R(a_d;\Phi))$, where
\begin{align*}
{\mathcal S}_M(a_d;\Phi) =& {\mathcal S}_{M,X,Y}(a_d;\Phi)
=\frac{1}{X} \sum_{(d,2)=1} M_Y(d) a_d \Phi\biggl(
\frac {N(d)}X\biggr), \\
{\mathcal S}_R(a_d;\Phi) =& {\mathcal S}_{R,X,Y}(a_d;\Phi) = \frac{1}{X} \sum_{(d,2)=1 }
\Big|R_Y(d) a_d
\Phi\Big(\frac{N(d)}{X}\Big)\Big|.
\end{align*}

  Let $\varepsilon>0$ be a small number and let $b, M$ be two parameters such that $b \in [\epsilon, 1-\epsilon]$ and $X^{\epsilon} \le M \le X$.  Also, let $P(x)$ be a polynomial satisfying $P(0)=P^{\prime}(0)=0$ and $P(b)=1$, $P^{\prime}(b)=0$. We define a sequence $\{\lam(n)\}_{n \in \mathcal O_K}$ such that when $n \equiv 1 \bmod (1+i)^3$  and $N(n) \le M$, we have
\begin{align}
\label{lambdan}
\lam(n) :=
\mu_{[i]}(n) Q\Big(\frac{\log (M/N(n))}{\log M}\Big)
:=
\begin{cases}
\mu_{[i]}(n) & \text {if  } N(n) \le M^{1-b}, \\
\mu_{[i]}(n) P(\frac{\log (M/N(n))}{\log M}) &\text{if } M^{1-b} \le N(n) \le M.\\
\end{cases}
\end{align}
  For other values of $n$, we define $\lam(n)=0$.
 The definition above then implies that $\lambda(n) \ll N(n)^{\epsilon}$ for all $n$.  We use the $\lam(n)$ to define for any odd $d \in \mathcal O_K$, the following mollifier function
\begin{align}
\label{mollifier}
M(s,d) = \sum_{\substack{N(n)\le M \\ n \equiv 1 \bmod (1+i)^3}} \frac{\lam(n)}{N(n)^s} \chi_{(1+i)^5d}(n).
\end{align}

  Further, we define for any complex number $\delta_1$,
$$
{\mathcal W}(\delta_1,\Phi) =
\frac{{\mathcal S}(|L(\tfrac 12+\delta_1,\chi_{(1+i)^5d})M(\tfrac 12
+\delta_1,d)|^2;\Phi)}{{\mathcal S}(1;\Phi)}.
$$We now apply Lemma \ref{Lemma 2.1} to the function
$f(s,d) := L(s,\chi_{(1+i)^5d}) M(s,d)$ with $W_0=\frac 12 -\frac{R}{\log X}$,
$H=\frac{S}{\log X}$, and $W_1=\sigma_0:=1+ 3 \log \log M/\log M$ and  $R, S \in (\epsilon, 1/\epsilon)$
to be chosen later, such that $f(s,d)$ has no zeros in $\Re(s)> \sigma_0$.
Arguing as in \cite[Section 2]{C&S}, we deduce that
\begin{align}
\label{Nbound}
{\mathcal N}(X,\Phi) \le
\frac{X{\mathcal S}(1;\Phi)}{8S \sinh (\frac{\pi R}{2S})}
\Big( J_1(X;\Phi) + J_2(X;\Phi)\Big) + \frac{X}{8S \sinh(\frac{\pi R}{2S}
)} {\mathcal S}(I(d);\Phi),
\end{align}
where
\begin{align}
\label{Jdef}
\begin{split}
J_1(X;\Phi) =& \int_0^S \cos \Big ( \frac{\pi t}{2S} \Big )
\log {\mathcal W}
\Big(-\frac{R}{\log X} + i\frac{t}{\log X};\Phi\Big) dt, \\
J_2(X;\Phi) =& \int_{-R}^{(\sigma_0-\frac 12)\log X}
\sinh \Big( \frac{\pi (x+R)}{2S}\Big)
\log {\mathcal W}\Big(\frac{x}{\log X} + i\frac{S}{\log X};\Phi\Big) dx, \\
I(d) =& -\Re
\int_{-S}^{S} \cos\Big(\pi\frac{(\sigma_0-1/2)\log X +R +it}{2iS}\Big)
\log f\Big(\sigma_0 +i\frac t{\log X},d\Big) dt,
\end{split}
\end{align}

  It remains to estimate the quantities on the right-hand side of \eqref{Nbound}.  We first note the following result concerning  ${\mathcal S}(I(d);\Phi)$, whose proof is given in Section \ref{sec: proof of Prop. 2.1}.
\begin{prop}
\label{Proposition 2.1}
 Let $\Phi$ be a smooth function
supported on $[1,2]$ such that $0 \leq \Phi(t) \ll 1$ and $\int_1^2 \Phi(t) dt \gg 1$.  For large $X>0$, $M\le \sqrt{X}, \sigma_0=1+ 3 \log \log M/\log M$ and $\epsilon \leq \Re(\delta_1)<3/4$,  we have
$$
{\mathcal W}(\delta_1,\Phi) = 1 + O(\Phi_{(2)} X^{\epsilon}
(M^{-2\Re(\delta_1) (1-b)}+ M^{(\frac 12-\Re(\delta_1))(1-b)}X^{-\frac 12}) ).
$$
Moreover, the function $f(s,d)=L(s,\chi_{(1+i)^5d}) M(s,d)$ has no zeros in $\Re(s) > \sigma_0$, and we have
$$
{\mathcal S}(I(d);\Phi) \ll \exp\Big( \pi \frac{(1/2+\epsilon)\log X}{2S}\Big)
M^{-(1-b)} X^{\epsilon}.
$$
\end{prop}

  Next, notice that an evaluation on the terms in \eqref{Nbound} involving with $J_1(X;\Phi)$ and $J_2(X;\Phi)$ requires one to study ${\mathcal W}(\delta_1,\Phi)$. As Proposition \ref{Proposition 2.1} enables us to treat ${\mathcal W}(\delta_1,\Phi)$ when $\Re(\delta_1)$ is slightly away from $0$, we may focus on the case when $\delta_1$ is near $0$. For this purpose,   we shall evaluate more generally the following expression:
\begin{align}
\label{Sgen}
{\mathcal S}(\xi(\tfrac 12+\delta_1,\chi_{(1+i)^5d}) \xi(\tfrac 12+\delta_2,\chi_{(1+i)^5d})
M(\tfrac 12+\delta_1,d)M(\tfrac 12+\delta_2,d);\Psi),
\end{align}
where $\xi$ is given in \eqref{xidef}, $\Psi$ is a smooth function supported on $[1,2]$ and $\delta_1, \delta_2$ satisfy the conditions given in Section \ref{sect: apprfcneqn}.
  Later, we shall set $\delta_2=\overline{\delta_1}$ and $\Psi(t)= \Phi(t)t^{-\tau}$ in \eqref{Sgen} to retrieve the expression
${\mathcal S}(1;\Phi)(2^5X/\pi^2)^{\tau}\Gamma_\delta(\tau){\mathcal W}(\delta_1,\Phi)$.

 In order to evaluate \eqref{Sgen}, we apply the approximate functional equation for $
\xi(\frac 12+\delta_1,\chi_{(1+i)^5d})\xi(\frac 12+\delta_2,\chi_{(1+i)^5d})$ to see that we may recast the expression in \eqref{Sgen} as
\begin{align*}
{\mathcal S}_M(A_{\delta,\tau}(d)M(\tfrac 12+\delta_1,d)M(\tfrac 12
+\delta_2,d);\Psi)+{\mathcal S}_R(A_{\delta,\tau}(d)M(\tfrac 12+\delta_1,d)M(\tfrac 12
+\delta_2,d);\Psi).
\end{align*}

  In Section \ref{sec: SR}, we obtain the following estimation for ${\mathcal S}_R$.
\begin{prop}
\label{Proposition 2.2}  With the above notations, we have for $M\le \sqrt{X}$,
$$
{\mathcal S}_R(A_{\delta,\tau}(d)M(\tfrac 12+\delta_1,d)M(\tfrac 12
+\delta_2,d);\Psi) \ll X^{\kappa+\epsilon}
\Big(\frac{1}{Y} +\frac{M^{-\Re(\delta_1)}+M^{-\Re(\delta_2)}+M^{-2\Re(\tau)}}
{Y^{\frac 12}} + \frac {M^{1-2\Re(\tau)}}{X^{\frac 12}}\Big).
$$
\end{prop}

  The treatment on ${\mathcal S}_M(A_{\delta,\tau}(d) M(\tfrac 12+
\delta_1,d)M(\frac 12+\delta_2,d);\Psi)$ is more involved.  Our proof in fact requires us to evaluate
more generally ${\mathcal S}_M(A_{\delta,\tau}(d) \leg{(1+i)^5d}{l};\Psi)$
for any odd primary $l \in \mathcal O_K$.  To state our result, we introduce a
few notations.  We define for any two complex numbers
$s$ and $w$,
\begin{align}
\label{Zswdef}
Z(s;w) = \zeta_K(s-2w)\zeta_K(s) \zeta_K(s+2w).
\end{align}

 Our evaluation on $S_M$ is given in the following result.
\begin{prop}
\label{Proposition 2.3}  With the above notations and writing any odd primary $l \in \mathcal O_K$ as $l=l_1l_2^2$, with $l_1$ being primary, square-free and $l_2 \in G $, we have
\begin{align}
\label{SMasymp}
\begin{split}
{\mathcal S}_M\Big(\leg{(1+i)^5d}{l} A_{\delta,\tau}(d);\Psi\Big)
=&  \frac{2}{3\zeta_K(2) \sqrt{N(l_1)}}
\sum_{\mu=\pm} \Big( r_\delta(l_1)\Gamma_\delta(\mu\tau) \leg{2^5X}{N(l_1)
\pi^2}^{\mu \tau} {\widehat \Psi}(\mu \tau+1) Z(1+2\mu\tau;\delta)
\eta_\delta(1+2\mu\tau;l)\\
&+ r_\tau(l_1) \Gamma_\tau(\mu\delta) \leg{2^5X}{N(l_1)\pi^2}^{\mu\delta} {\widehat \Psi}
(\mu\delta+1)
Z(1+2\mu\delta;\tau)\eta_\tau(1+2\mu\delta;l)\Big)
\\
&+ {\mathcal R}(l) + O\Big(\frac{|r_\delta(l_1)|X^{\epsilon}}{(XN(l_1))^{\frac 14}}
+\frac{N(l)^\epsilon X^{\kappa +\epsilon}
 N(l_1)^{2\kappa-\frac 12}}{Y^{1-4\kappa}}\Big).
\end{split}
\end{align}
Here ${\mathcal R}(l)$ is a remainder term bounded on average by
\begin{align}
\label{Rbound}
\sum_{\substack{l \equiv 1 \bmod (1+i)^3 \\ L \leq N(l) \leq 2L-1}} |{\mathcal R}(l)| \ll \Big ( \frac{L^{1+\epsilon}Y^{1+\epsilon}}
{X^{\frac 12-|\Re(\delta)|-\epsilon} } +\frac{L^{1+\kappa+\epsilon}Y^{2\kappa+\epsilon}}
{X^{\frac 12-|\Re(\delta)|-\epsilon} } \Big ) \Psi_{(3)}\Psi_{(4)}^{\epsilon}.
\end{align}
\end{prop}

   With Propositions \ref{Proposition 2.1}-\ref{Proposition 2.3} available, we are able to obtain the following result concerning ${\mathcal W}(\delta_1,\Phi)$ for $|\delta_1|$ being small.
\begin{proposition}
\label{Proposition 2.4}  Let $\Phi$ be a non-negative
smooth function on $[1,2]$ satisfying $\Phi(t) \ll 1$ and $\int_1^2 \Phi(t) dt\gg 1$.
Let $\delta_1$ be a complex number such
that $\Re(\delta_1) \ge -\frac{1}{\epsilon \log X}$, $ \kappa  \geq |\delta_1|
\ge \frac{\epsilon}{\log X}$.  We take $\delta_2=\overline{\delta_1}$
so that $\tau= \Re(\delta_1)$, and $\delta= i\Im(\delta_1)$.
Then with the mollifier function being given in \eqref{mollifier}, we have for $M=X^{\frac 12-5\kappa}$,
\begin{align*}
{\mathcal W}(\delta_1,\Phi)=& 1+ \Big( \frac{1-(2^5X/\pi^2)^{-2\tau}}{2\tau \log M}
-\leg{2^5X}{\pi^2}^{-\tau} \frac{(2^5X/\pi^2)^{\delta}-(2^5X/\pi^2)^{-\delta}}{2\delta
\log M}\Big)
\int_0^b M^{-2\tau(1-x)} \Big| Q^{\prime}(x) + \frac{Q^{\prime \prime}(x)}{
2\delta_1 \log M}\Big|^2 dx \\
&+O(X^{-\kappa-\epsilon} \Phi_{(3)}
\Phi_{(4)}^{\epsilon}+M^{-2\tau(1-b)}|\delta_1|^6 \log^5 X).
\end{align*}
\end{proposition}

  The proofs of Propositions \ref{Proposition 2.1}-\ref{Proposition 2.4} will occupy major parts in the rest of the paper. In the end, the proof of Theorem \ref{thm: nonvanishing} follows by gathering these results together with suitable choices for the parameters involved.

\section{Proofs of Propositions \ref{Proposition 2.1} and \ref{Proposition 2.2}}

\subsection{Proof of Proposition \ref{Proposition 2.1}}
\label{sec: proof of Prop. 2.1}

   In this section we estimate ${\mathcal S}(I(d);\Phi)$ by proving Proposition \ref{Proposition 2.1}. We note first that similar to the bounds given in \cite[(4.1), (4.2)]{C&S},  it follows from
Lemma \ref{lem: estimates for character sums} that for any odd $l \in \mathcal O_K$ such that $N(l) \le \sqrt{2X}$, we have
\begin{align}
\label{Mest}
\begin{split}
\sum_{X\le N(d)\le 2X} \mu^2_{[i]}((1+i)d) |M(s,d)|^4
\ll & X^{\epsilon} (X+ XM^{2(1-2\Re(s))} +M^{4(1-\Re(s))}),  \\
\sum_{X/N(l)^2 \le N(m) \le 2X/N(l)^2} \mu^2_{[i]}((1+i)m) |M(s,l^2m)|^4
\ll&  X^{\epsilon} \Big(\frac{X}{N(l)^2} + \frac{X}{N(l)^2} M^{2(1-2\Re(s))}
+ M^{4(1-\Re(s))}\Big).
\end{split}
\end{align}

  Now, we write $B(s,d) = L(s,\chi_{(1+i)^5d}) M(s,d) -1$ and we observe that it follows from \cite[Section 3.1]{G&Zhao4} and partial summation that we have
\begin{align}
\label{S1}
\begin{split}
{\mathcal S}(1; \Phi)=\frac{2\pi{\widehat \Phi}(1)}{3\zeta_K(2)}+O(X^{-1/2}) \gg 1.
\end{split}
\end{align}
  Using the above, we see that
\begin{align}
\label{W-1}
{\mathcal W}(\delta_1,\Phi) = 1+ O({\mathcal S}(B(\tfrac 12+\delta_1,d);\Phi)
+ {\mathcal S}(|B(\tfrac 12+\delta_1,d)|^2;\Phi)). \
\end{align}
To estimate the error terms above, for any real number $c> \frac 12 -\Re(\delta_1)$, consider the integral
$$
\frac{1}{2\pi i} \int\limits_{(c)} \Gamma(s) B(\tfrac 12+\delta_1+s,d) X^s
ds.
$$
 By moving the line of integration above to $\Re(s) = -\Re(\delta_1)$, we see that the
pole at $s=0$ contributes $B(\frac 12+\delta_1;d)$. This implies that
\begin{align*}
B(\tfrac 12+\delta_1,d) =& \frac{1}{2\pi i} \int\limits_{(c)}
\Gamma(s) B(\tfrac 12+\delta_1+s,d) X^s ds
- \frac{1}{2\pi i}\int\limits_{(-\Re(\delta_1))}
\Gamma(s) B(\tfrac 12+\delta_1+s,d) X^s ds \\
 :=&  T_1(\frac 12+\delta_1,d)-
T_2(\frac 12+\delta_1,d) .
\end{align*}

We apply the Cauchy-Schwarz inequality to see that
$$
|T_2(\tfrac 12+\delta_1,d)|^2 \ll X^{-2\Re(\delta_1)}
\Big( \int\limits_{(-\Re(\delta_1))} |\Gamma(s)B(\tfrac 12+\delta_1+s,d)^2 ds|
\Big)\Big(\int\limits_{(-\Re(\delta_1))} |\Gamma(s) ds|\Big),
$$
It follows from this and the rapid decay of $|\Gamma(s)|$ when $|\Im(s)| \to
\infty$ that we have
$$
|T_2(\tfrac 12+\delta_1,d)|^2 \ll X^{-2\Re(\delta_1)}
\Big( 1+ \int\limits_{(-\Re(\delta_1))} |\Gamma(s)| |L(\tfrac 12+\delta_1+s,
\chi_{-8d})
M(\tfrac 12+\delta_1+s,d)|^2 |ds|\Big).
$$
Applying the Cauchy-Schwarz inequality again, we see that
$$
 {\mathcal S}(|T_2(\frac 12+\delta_1,d)|^2;\Phi) \ll X^{-2\Re(\delta_1)} \Big( 1+
 \int\limits_{(-\Re(\delta_1))} |\Gamma(s)| {\mathcal S}(|L(\tfrac 12+\delta_1+s,
\chi_{-8d})|^4;\Phi)^{\frac 12}
{\mathcal S}(|M(\tfrac 12+\delta_1+s,d)|^4;\Phi)^{\frac 12} |ds| \Big),
$$
 Applying Lemma \ref{lem:2.3} and \eqref{Mest} in the above estimation, we see that for $M\le \sqrt{X}$,
\begin{align}
\label{ST2}
{\mathcal S}(|T_2(\tfrac 12+\delta_1, d)|^2; \Phi) \ll
X^{-2\Re(\delta_1) +\epsilon}.
\end{align}
 With one more application of the Cauchy-Schwarz inequality, we deduce from the above that
\begin{align}
\label{ST2-1}
  {\mathcal S}(|T_2(\tfrac 12+\delta_1,d)|;\Phi)
\ll X^{-\Re(\delta_1) +\epsilon}.
\end{align}

  Now, we consider the contribution from $T_1$.  For $\Re(s) >1$,  we write
$$
B(s,d) = \sum_{\substack{n \equiv 1 \bmod (1+i)^3}} \frac{b(n)}{N(n)^s} \leg {(1+i)^5d}{n}.
$$
 By \eqref{lambdan} and \eqref{mollifier}, we see that
$b(n)=0$ for all $N(n) \le M^{1-b}$ and that $|b(n)|\ll N(n)^{\epsilon} d_{[i]}(n)\ll N(n)^{\epsilon}$
for all $n$,  where we denote $d_{[i]}(n)$ for the analogue on $\mathcal{O}_K$ of the usual divisor function $d$ on $\mz$. Moreover, as $\lam$ is supported
on square-free numbers, we have $b(m^2) = \sum_{\substack{d|m^2 \\ d \equiv 1 \bmod (1+i)^3}}
\lam(d) =\sum_{\substack{d|m \\ d \equiv 1 \bmod (1+i)^3}}\lam(d)=b(m)$. It follows that we have $b(n)=0$ for all primary
square values $N(n) \le M^{2(1-b)}$.   Now, we have
\begin{align}
\label{T1int}
\begin{split}
T_1(\tfrac 12+\delta_1,d) =& \frac{1}{2\pi i} \int\limits_{(c)} \Gamma(s)
\sum_{\substack{n \equiv 1 \bmod (1+i)^3 \\ M^{1-b} \le N(n) \le X\log^2 X}} \frac{b(n)}{N(n)^{\frac 12+\delta_1+s}}
\leg{(1+i)^5d}{n} X^s ds \\
&+ \sum_{\substack{n \equiv 1 \bmod (1+i)^3 \\ N(n)> X\log^2 X }} \frac{b(n)}{N(n)^{\frac 12+\delta_1}} \leg{(1+i)^5d}{n}
\Big(\frac{1}{2\pi i}\int\limits_{(c)} \Gamma(s) \leg{X}{N(n)}^s
ds \Big).
\end{split}
\end{align}
 Using the result that
$$\frac{1}{2\pi i} \int\limits_{(c)} \Gamma(s) \leg{X}{N(n)}^s ds
=\exp(-N(n)/(20X)),$$
 we see that the second term on the right side of \eqref{T1int} contributes $\ll X^{-5}$.  We then move the
line of integration in the first term on the right side of \eqref{T1int}to $\Re(s)=\frac 1{\log X}$  to see that
\begin{align}
\label{T1}
T_1(\tfrac 12+\delta_1,d) =
 \frac{1}{2\pi i} \int\limits_{(\frac{1}{\log X})}
\Gamma(s)
\sum_{\substack{n \equiv 1 \bmod (1+i)^3 \\ M^{1-b} \le N(n) \le X\log^2 X}} \frac{b(n)}{N(n)^{\frac 12+\delta_1+s}}
\leg{(1+i)^5d}{n} X^s ds +O(X^{-5}).
\end{align}
 It follows from this and the Cauchy-Schwarz inequality that we have
\begin{align*}
|T_1(\tfrac 12+\delta_1,d)|^2 &\ll X^{-10} +
\Big(\int\limits_{(\frac 1{\log X})} |\Gamma(s)|\Big| \sum_{\substack{n \equiv 1 \bmod (1+i)^3 \\ M^{1-b} \le N(n) \le X\log^2 X}} \frac{b(n)}{N(n)^{\frac 12+\delta_1+s}} \leg{(1+i)^5d}{n}\Big|^2 |ds|
\Big)\Big(\int\limits_{(\frac 1{\log X})} |\Gamma(s)ds| \Big)\\
&\ll X^{-10} + X^{\epsilon}
\int\limits_{(\frac 1{\log X})} |\Gamma(s)|\Big| \sum_{\substack{n \equiv 1 \bmod (1+i)^3 \\ M^{1-b} \le N(n) \le X\log^2 X}} \frac{b(n)}{N(n)^{\frac 12+\delta_1+s}} \leg{(1+i)^5d}{n}\Big|^2 |ds|.
\end{align*}

  We split the sum over $n$ above into dyadic blocks and apply Lemma
 \ref{lem: estimates for character sums} to see that
$$
{\mathcal S}(|T_1(\tfrac 12+\delta_1,d)|^2;\Phi)
\ll  M^{-2 \Re(\delta_1)(1-b)}X^{\epsilon}.
$$
 Combining this with \eqref{ST2}, we deduce that
\begin{align}
\label{B2}
{\mathcal S}(|B(\tfrac 12+\delta_1,d)|^2;\Phi) \ll M^{-2\Re(\delta_1)(1-b)}
X^{\epsilon}.
\end{align}

 Next, we bound
${\mathcal S}(T_1(\tfrac 12+\delta_1,d);\Phi)$ by applying \eqref{T1} to see that
\begin{align}
\label{ST1}
{\mathcal S}(T_1(\tfrac 12+\delta_1,d);\Phi) \ll
X^{-5} + X^{\epsilon} \sum_{\substack{n \equiv 1 \bmod (1+i)^3 \\ M^{1-b} \le N(n) \le X\log^2 X}}  \frac{|b(n)|}
{N(n)^{\frac 12+\Re(\delta_1)}}
\Big|{\mathcal S}\big( \chi_{(1+i)^5d}(n) ;\Phi\big)\Big|.
\end{align}
We now apply the Mellin transform to see that for any $c>1$,
\begin{align*}
{\mathcal S}\big( \chi_{(1+i)^5d}(n) ;\Phi\big) =& \frac{\leg {(1+i)^5}{n} }{2\pi i}
\int\limits_{(c)} \sum_{d \equiv 1 \bmod (1+i)^3} \frac{\mu^2_{[i]} ((1+i)d) \chi_n(d) }{N(d)^w}
X^{w-1} {\widehat \Phi}(w) dw \\
=&  \frac{\leg {(1+i)^5}{n}}{2\pi i}
\int\limits_{(c)} \frac{L(w,\chi_n)}{L(2w,\chi_n)} (1+ \frac {\chi_n(1+i)}{2^w})^{-1}
X^{w-1} {\widehat \Phi}(w)dw.
\end{align*}
  Moving the line of integration
above to $\Re(w) = \frac 12+ \frac{1}{\log X}$, we see that we
encounter a pole at $w=1$ if and only if $n=\square$ and the corresponding contribution of the
residue is $\ll 1$.  It follows that
$$
|{\mathcal S}\big( \chi_{(1+i)^5d}(n) ;\Phi\big)| \ll \delta(n=\square) + X^{-\frac 12 +
\epsilon}
\int\limits_{(\frac 12+\frac
{1}{\log X})} |L(w,\chi_n)| |{\widehat \Phi}(w)| |dw|,
$$
where $\delta(n=\square)$ is $1$ if $n=\square$ and $0$ otherwise.
As $b(n)=0$ for all perfect squares with norm $\le M^{2(1-b)}$, we
deduce that
\begin{align}
\label{4.8}
\begin{split}
& \sum_{\substack{n \equiv 1 \bmod (1+i)^3 \\ M^{1-b} \le N(n) \le X\log^2 X}} \frac{|b(n)|}
{N(n)^{\frac 12+\Re(\delta_1)}}
\Big|{\mathcal S}\big( \chi_{(1+i)^5d}(n) ;\Phi\big)\Big| \\
\ll & X^{\epsilon}  M^{-2\Re(\delta_1)(1-b)} +
X^{-\frac 12+\epsilon} \int\limits_{(\frac 12+\frac
{1}{\log X})} \sum_{\substack{n \equiv 1 \bmod (1+i)^3 \\ M^{1-b} \le N(n) \le X\log^2 X}}  \frac{1}
{N(n)^{\frac 12+\Re(\delta_1)}}  |L(w,\chi_n)| |{\widehat \Phi}(w)||dw|.
\end{split}
\end{align}
 Note that Lemma \ref{lem:2.3} implies that
$$
\sum_{\substack{n \equiv 1 \bmod (1+i)^3 \\ N\le N(n)\le 2N}} |L(w,\chi_n)| \ll N^{1+\epsilon} (1+|w|^2)^{\frac 14
+\epsilon}.
$$
Applying this together with \eqref{gsest} by setting $g=\Phi$ and $\nu=2$ there,
we obtain from \eqref{4.8} that
\begin{align*}
\begin{split}
& \sum_{\substack{n \equiv 1 \bmod (1+i)^3 \\ M^{1-b} \le N(n) \le X\log^2 X}} \frac{|b(n)|}
{N(n)^{\frac 12+\Re(\delta_1)}}
\Big|{\mathcal S}\big( \chi_{(1+i)^5d}(n) ;\Phi\big)\Big|
\ll
X^{\epsilon} \Phi_{(2)}
\Big( M^{-2\Re(\delta_1)(1-b)} + X^{-\Re(\delta_1)}
+ M^{(\frac 12-\Re(\delta_1))(1-b)}X^{-\frac 12}\Big).
\end{split}
\end{align*}

 We insert the above in \eqref{ST1} and further combine the result with \eqref{ST2-1} to
deduce that for $M \le \sqrt{X}$,
\begin{align}
\label{B1}
{\mathcal S}(B(\frac 12+\delta_1,d);\Phi)
\ll X^{\epsilon} \Phi_{(2)}
\Big( M^{-2\Re(\delta_1)(1-b)}
+ M^{(\frac 12-\Re(\delta_1))(1-b)}X^{-\frac 12}\Big).
\end{align}
 The first statement of Proposition \ref{Proposition 2.1} now follows by applying \eqref{B2} and \eqref{B1} in \eqref{W-1}.

To prove the second assertion, we note that
$$
f(s,d) = 1+B(s,d) = 1+ O\Big( \sum_{\substack{ n \equiv 1 \bmod (1+i)^3 \\ N(n)\ge M^{1-b}}} \frac{d_{[i]}(n)}{N(n)^{\Re(s)}}
\Big).
$$
We deduce readily from this that  $f(s,d)$ has no zeros to the
right of $\sigma_0$. Moreover, we have $\log f(s,d) = B(s,d) + O(|B(s,d)|^2)$
for $\Re(s) > \sigma_0$.  Thus, we obtain that
$$
{\mathcal S}(I(d);\Phi) \ll \exp\Big(\pi \frac{(\frac 12+\epsilon)\log X}{2S}
\Big) \Big(  |{\mathcal S}(B(s,d);\Phi)| +{\mathcal S}(|B(s,d)|^2;\Phi)\Big).
$$
  By applying \eqref{B2} and \eqref{B1}, we see that the second statement of Proposition \ref{Proposition 2.1} now follows.

\subsection{Proof of Proposition \ref{Proposition 2.2}}
\label{sec: SR}

  In this section, we estimate ${\mathcal S}_R(A_{\delta,\tau}(d)M(\tfrac 12+\delta_1,d)M(\tfrac 12
+\delta_2,d);\Psi)$ by proving Proposition \ref{Proposition 2.2}. We first notice that $|R_Y(d)| \le \sum_{k|d} 1
\ll N(d)^{\epsilon}$ and that $R_Y(d)=0$ unless $d=l^2m$ where $m$
is square-free and $N(l)>Y$. It follows that
\begin{align}
\label{SRbound}
\begin{split}
& {\mathcal S}_R(A_{\delta,\tau}(d) M(\tfrac 12+\delta_1,d)M(\tfrac 12+\delta_2,d)
;\Psi)\\
\ll & X^{-1+\epsilon} \sum_{\substack{ N(l)>Y \\ (l,2)=1}} \sumflat_{X/N(l)^2 \le N(m) \le 2 X/N(l)^2}
|A_{\delta,\tau}(l^2m) M(\tfrac 12+\delta_1,l^2m)M(\tfrac 12+\delta_2,l^2m)|,
\end{split}
\end{align}
where $\sum^{\flat}$ means that the sum is over odd and square-free $m \in \mathcal O_K$.  Now, applying twice the Cauchy-Schwarz inequality, we see that the sum over $m$ above is
\begin{align}
\label{msumbound}
\ll \biggl(\sumflat_{ m}|M(\tfrac 12+\delta_1,l^2m)|^{4}
\biggr)^{\frac14}
 \biggl(\sumflat_{m }|M(\tfrac 12+\delta_2,l^2m)|^{4}
\biggr)^{\frac14}
\biggl(\sumflat_{m} |A_{\delta,\tau}(l^2m)|^2
\biggr)^{\frac 12}.
\end{align}

Note that for any $c>\frac 12 +|\Re(\delta)|$, we have
\begin{align}
\label{Aintrep}
A_{\delta,\tau}(l^2 m)
=& \frac {1}{2\pi i} \int\limits_{(c)}
\Gamma_\delta(s)
\leg{2^5N(l^2 m)}{\pi^2}^{s} \frac{2s}{s^2 -\tau^2}
\sum_{n \equiv 1 \bmod (1+i)^3} \frac{r_\delta(n)}{N(n)^{s+\frac 12}}\leg {(1+i)^5l^2 m}{n}
ds.
\end{align}

We write the sum above as
\begin{align}
\label{Ldecomp}
\sum_{n \equiv 1 \bmod (1+i)^3}  \frac{r_\delta(n)}{N(n)^{s+\frac 12}}\leg {(1+i)^5l^2 m}{n} =
L(\tfrac 12+s+\delta,\chi_{(1+i)^5m})L(\tfrac 12 + s-\delta,\chi_{(1+i)^5m})
 {\mathcal E}(s,l),
\end{align}
with
$$
{\mathcal E}(s,l)= \prod_{\substack{\varpi \in G \\ \varpi|l}} \biggl(1-\frac{1}{N(\varpi)^{s+\frac 12+\delta}}
\leg{(1+i)^5m}{\varpi}\biggr)\biggl(1-\frac{1}{N(\varpi)^{s+\frac 12-\delta}}\leg {(1+i)^5 m}{\varpi}\biggr).
$$
Observe that the left side of \eqref{Ldecomp} is analytic for all $s$ since $\chi_{(1+i)^5m}$ is non-principal. Thus, we
we may move the line of integration in \eqref{Aintrep} to $\Re(s)=\kappa+ 1/\log X$ without encountering any pole.  Using the estimations that
$|{\mathcal E}(s,l)|\le \prod_{\substack{\varpi \in G \\ \varpi|l}} (1+1/\sqrt{N(\varpi)})^2 \ll N(l)^{\epsilon}
\ll X^{\epsilon}$,  $2s/(s^2 -\tau^2) \ll X^\epsilon$, and the rapid decay of
$|\Gamma_\delta(s)|$ when $|\Im(s)| \rightarrow \infty$, we apply the Cauchy-Schwarz inequality to see that
$$
|A_{\delta,\tau}(l^2m)|^2 \ll X^{2\kappa+\epsilon}
\int\limits_{(\kappa+\frac 1{\log X})}
|\Gamma_\delta(s)| |L(\tfrac 12+s+\delta, \chi_{(1+i)^5m})L(\tfrac 12+s-\delta,
\chi_{(1+i)^5m})|^{2}
|ds|.
$$
It follows from this and Lemma \ref{lem:2.3} that we have
$$
\sumflat_{X/N(l)^2 \le N(m)\le 2X/N(l)^2}
|A_{\delta,\tau}(l^2 m)|^2
\ll \frac{X^{1+2\kappa+\epsilon}}{N(l)^2}
\int\limits_{(\kappa+\frac 1{\log X})} |\Gamma_\delta(s)|
(1+|s|^2)^{1+\epsilon} |ds| \ll \frac{X^{1+2\kappa+\epsilon}}{N(l)^2}.
$$
 We apply this with \eqref{Mest} to see that the expression in \eqref{msumbound} is bounded by
$$
\ll \frac {X^{1/2+\kappa +\epsilon}}{N(l)} \Big(\frac{X^{1/4}}{N(l)^{1/2}} +\frac{X^{1/4}}{N(l)^{1/2}}M^{-\Re(\delta_1)}+M^{1/2-\Re(\delta_1)} \Big )\Big(\frac{X^{1/4}}{N(l)^{1/2}} +\frac{X^{1/4}}{N(l)^{1/2}}M^{-\Re(\delta_2)}+M^{1/2-\Re(\delta_2)} \Big ).
$$
Inserting the above in \eqref{SRbound} and keeping in mind that $M \leq \sqrt{X}$, we see that the assertion of Proposition \ref{Proposition 2.2} follows.

\section{Proof of Proposition \ref{Proposition 2.3}}

\subsection{A first decomposition}

  Note that we have
\begin{align}
\label{SMseries}
{\mathcal S}_M\biggl(\leg{(1+i)^5d}{l} A_{\delta,\tau}(d) ;\Psi\biggr)
=  \sum_{n \equiv 1 \bmod (1+i)^3} \frac{r_{\delta}(n)}
{\sqrt{N(n)}} {\mathcal S}_M \biggl(\leg{(1+i)^5d}{ln};F_{N(n)}\biggr),
\end{align}
 where $F_{N(n)}(t)$ is defined as in \eqref{Fdef}.

  By definition, we have
\begin{align*}
   {\mathcal S}_M\biggl(\leg{(1+i)^5d}{ln};F_{N(n)} \biggr) =& \frac 1X\sum_{(d,2)=1}\Big ( \sum_{\substack{\alpha \equiv 1 \bmod (1+i)^3 \\ N(\alpha) \le Y\\ \alpha^2 |d }} \mu_{[i]}(\alpha) \Big )\leg{(1+i)^5d}{ln}F_{N(n)}\leg {N(d)}{X} \\
  =& \frac 1X \sum_{\substack{\alpha \equiv 1 \bmod (1+i)^3 \\ N(\alpha) \le Y}} \mu_{[i]}(\alpha) \sum_{(d,2)=1}\leg{(1+i)^5 \alpha^2 d}{ln}F_{N(n)}\leg {N(\alpha^2 d)}{X}.
\end{align*}
 We apply the Poisson summation formula given in Lemma \ref{Poissonsumformodd} to treat the last sum above to obtain that
\begin{align}
\label{SMPS}
{\mathcal S}_M\biggl(\leg{(1+i)^5d}{ln};F_{N(n)}\biggr) = \frac{1}{2N(ln)} \leg{(1+i)^6}{ln}
 \sum_{\substack{\alpha \equiv 1 \bmod (1+i)^3 \\ N(\alpha) \le Y\\ (\alpha,ln)=1 }} \frac{\mu_{[i]}(\alpha)}{N(\alpha)^2}
\sum_{k \in \mathcal O_K} (-1)^{N(k)} g(k,ln) {\widetilde F}_{N(n)}\Big ( \sqrt{ \frac{N(k)X}{2N(\alpha^2ln)}} \Big ).
\end{align}

The above allows us to deduce from \eqref{SMseries} that
\begin{align}
\label{SMdecomp}
{\mathcal S}_M\biggl(\leg{(1+i)^5d}{l}A_{\delta,\tau}(d)
;\Psi\biggr)
=  {\mathcal P}(l) + {\mathcal R}_{0}(l),
\end{align}
where ${\mathcal P}(l)$ arises from the $k=0$ term in \eqref{SMPS} and ${\mathcal R}_{0}(l)$
includes the remaining non-zero terms $k$ in
\eqref{SMPS}. Hence
\begin{align}
\label{PR}
\begin{split}
{\mathcal P}(l) =& \frac{1}{2N(l)}\sum_{n \equiv 1 \bmod (1+i)^3}
\frac{r_{\delta}(n)}{N(n)^{\frac 32}} \leg{(1+i)^6}{ln}
\sum_{\substack{\alpha \equiv 1 \bmod (1+i)^3 \\  N(\alpha) \le Y\\ (\alpha,ln)
=1 }} \frac{\mu_{[i]}(\alpha)}{N(\alpha)^2}  g(0, ln){\widetilde F_{N(n)}}(0), \\
{\mathcal R}_{0}(l)
=& \frac{1}{2N(l)}\sum_{n \equiv 1 \bmod (1+i)^3} \frac{r_{\delta}(n)}
{ N(n)^{\frac 3 2}}
\leg {(1+i)^6}{l n} \sum_{\substack{\alpha \equiv 1 \bmod (1+i)^3 \\  N(\alpha) \le Y\\ (\alpha,  l n) =1}}
\frac{\mu_{[i]}(\alpha)}{N(\alpha)^2}
\sum_{\substack{ k \in \mathcal O_K \\ k\neq 0 }}
 (-1)^{N(k)} g(k,ln) {\widetilde F}_{N(n)} \Big ( \sqrt{ \frac{N(k)X}{2N(\alpha^2ln)}} \Big ).
\end{split}
\end{align}

   We show in what follows that ${\mathcal P}(l)$ contributes to a main term and ${\mathcal R}_{0}(l)$ also contributes to secondary main terms.

\subsection{The principal term ${\mathcal P}(l)$}

 Note that $g(0, ln)=\varphi_{[i]}(ln)$ if $ln=\square$ and $g(0,ln)=0$ otherwise.
Also, we have
$$
\sum_{\substack{\alpha \equiv 1 \bmod (1+i)^3 \\  N(\alpha) \le Y\\ (\alpha,ln)
=1 }} \frac{\mu_{[i]}(\alpha)}{N(\alpha)^2}
= \frac{1}{\zeta_K(2)} \prod_{\substack{\varpi \in G \\ \varpi |2ln}} \Big( 1-\frac 1{N(\varpi)^2}\Big)^{-1}
\Big( 1+ O\Big(\frac 1Y\Big)\Big).
$$
 Using the above observations, we see that
$$
{\mathcal P}(l)
= \frac{1+O(Y^{-1})}{\zeta_K(2)}
\sum_{\substack{ n \equiv 1 \bmod (1+i)^3 \\ ln=\square }}
\frac{r_{\delta}(n)}{N(n)^{\frac12}} \leg{(1+i)^6}{ln}
\prod_{\substack{\varpi \in G \\ \varpi |2ln}} \leg{N(\varpi)}{N(\varpi)+1} {\widetilde F}_{N(n)}(0).
$$

As $l = l_1 l_2^2$ with $l_1$ primary and square-free, we see that $l n =\square$ is
equivalent to $ n= l_1 m^2$ for some primary $m$.
Thus
\begin{align*}
{\mathcal P}(l)
&= \frac{1+O(Y^{-1})}{\zeta_K(2)\sqrt{N(l_1)}}
\sum_{\substack{ m \equiv 1 \bmod (1+i)^3 }}
\frac{r_{\delta}(l_1m^2)}{N(m)}
\prod_{\substack{\varpi \in G \\ \varpi|2lm}} \leg{N(\varpi)}{N(\varpi)+1} {\widetilde  F}_{N(l_1m^2)}(0).
\end{align*}

 Note that we have for any $c>|\Re(\tau)|$,
\begin{align*}
\begin{split}
  \widetilde{F}_{N(l_1m^2)} \left(0\right) =& \int\limits^{\infty}_{-\infty}\int\limits^{\infty}_{-\infty}\Psi \left(N(x+yi) \right)
W_{\delta,\tau} \left(  \frac {\pi^2N(l_1m^2)}{2^5 X N(x+yi)} \right ) \dif x \dif y \\
=&  \frac {1}{2\pi
   i} \int\limits\limits_{(c)}\Gamma_\delta(s) \left(\frac{2^{5/2}}{\pi}\right)^{2s}
\Big (\frac {X}{N(l_1m^2)}\Big )^{s}
 \left (\int\limits^{\infty}_{-\infty}\int\limits^{\infty}_{-\infty}\Psi \left(N(x+yi) \right) N(x+yi)^{s}
\dif x \dif y \right )  \frac{2s}{s^2 -\tau^2}ds \\
=& \frac {\pi }{2\pi
   i} \int\limits\limits\limits_{(c)} \Gamma_\delta(s)\left(\frac{2^{5/2}}{\pi}\right)^{2s}
\Big (\frac {X}{N(l_1m^2)}\Big )^{s}
 \widehat{\Psi}(1+s)\frac{2s}{s^2 -\tau^2}ds,
\end{split}
\end{align*}
  where we deduce the last equality above by noticing that
\begin{align*}
  \int\limits^{\infty}_{-\infty}\int\limits^{\infty}_{-\infty}\Psi \left(N(x+yi) \right) N(x+yi)^{s}
\dif x \dif y =\int^{2\pi}_0\int^{\infty}_0\Psi (r^2)r^{2s}rdrd\theta =\pi \widehat{\Psi}(1+s).
\end{align*}

 We then conclude that for any $c>\kappa$,
\begin{align}
\label{Plfirstexp}
{\mathcal P}(l) =\frac {2\pi}3 \cdot \frac{1+O(Y^{-1})}{\zeta_K(2)\sqrt{N(l_1)}} I(l),
\end{align}
where
\begin{align}
\label{Il}
I(l)=\frac{1}{2\pi i} \int\limits_{(c)} \Gamma_{\delta}(s)
\left(\frac{2^{5/2}}{\pi}\right)^{2s}
\Big (\frac {X}{N(l_1)}\Big )^{s}\widehat{\Psi}(1+s)\frac{2s}{s^2 -\tau^2}
\sum_{\substack{ m \equiv 1 \bmod (1+i)^3 }}
\frac{r_{\delta}(l_1m^2)}{N(m)^{1+2s}}
\prod_{\substack{\varpi \in G \\ \varpi|lm}} \leg{N(\varpi)}{N(\varpi)+1}ds.
\end{align}

  Now, by comparing Euler factors, we see that for $\Re(s)>1 + 2|\Re(\delta)|$,
$$
\sum_{\substack{ m \equiv 1 \bmod (1+i)^3 }}
\frac{r_{\delta}(l_1m^2)}{N(m)^{s}}
\prod_{\substack{\varpi \in G \\ \varpi|lm}} \leg{N(\varpi)}{N(\varpi)+1}
= r_{\delta}(l_1) Z(s;\delta)
\eta_{\delta}(s;l),
$$
where $Z$ and $\eta$ are as defined in \eqref{Zswdef} and \eqref{etadef}. Using this in \eqref{Il}, we see that
\begin{align}
\label{I}
I(l)&= \frac{r_\delta(l_1)}{2\pi i}
\int\limits_{(c)} \Gamma_{\delta}(s)
\left(\frac{2^{5/2}}{\pi}\right)^{2s}
\Big (\frac {X}{N(l_1)}\Big )^{s}\widehat{\Psi}(1+s)\frac{2s}{s^2 -\tau^2}
Z(1+2s;\delta)\eta_{\delta}(1+2s;l)
ds.
\end{align}

Taking $c=\kappa+\epsilon$ above implies that
$I(l) \ll |r_\delta(l_1)| (X/N(l_1))^{\kappa+\epsilon}$.
We now move the line of integration in \eqref{I} to
$\Re s=-\frac 14 +\epsilon$ to encounter simple poles at $s=\pm \tau$, $\pm
\delta$ in the process. We note the following convexity bounds for $\zeta_K(s)$ from \cite[Exercise 3, p. 100]{iwakow}:
\begin{align}
\label{zetaconvexitybounds}
\begin{split}
 \zeta_K(s) \ll (1+|s|^2)^{\frac {1-\Re(s)}{2}+\varepsilon}, \quad 0< \Re(s)<1.
\end{split}
\end{align}
 It follows that $|Z(1+2s;\delta)| \ll (1+|s|^2)^3$ on the new line. Moreover, we have
$$|\eta_{\delta}(1+2s;l)|
\ll \prod_{\substack{\varpi \in G \\ \varpi|l_1}}(1+O(\frac{1}{\sqrt{N(\varpi)}})) \prod_{\substack{\varpi \in G \\ \varpi \nmid l_1}}
(1+O(\frac 1{N(\varpi)^{1+\epsilon}})) \ll N(l_1)^{\epsilon}.$$
Thus, we deduce that the integral on the new line is
$$
\ll \frac{|r_\delta(l_1)|N(l_1)^{\frac 14+\epsilon}}{X^{\frac 14 -\epsilon}}
\int\limits_{(-\frac 14 +\epsilon)}
|s|^5
|{\widehat \Psi}(1+s)| \Gamma_{\delta}(s)| |ds|
\ll  \frac{|r_\delta(l_1)|
N(l_1)^{\frac 14+\epsilon}}{X^{\frac 14-\epsilon}}.
$$
It follows  that
\begin{align*}
I(l)&=  r_\delta(l_1)
\mathop{\text{Res}}_{s=\pm \delta, \pm \tau}
\biggl\{ \Gamma_{\delta}(s)
\left(\frac{2^{5/2}}{\pi}\right)^{2s}
\Big (\frac {X}{N(l_1)}\Big )^{s}\widehat{\Psi}(1+s)\frac{2s}{s^2 -\tau^2}
Z(1+2s;\delta)\eta_{\delta}(1+2s;l)
\biggr\}+
O\biggl( \frac{|r_\delta(l_1)|N(l_1)^{\frac 14+\epsilon}}{X^{\frac 14
-\epsilon}} \biggr).
\end{align*}

  Applying this in \eqref{Plfirstexp}, we obtain that
\begin{align}
\label{Pl}
\begin{split}
{\mathcal P}(l)=&  \frac{2 \pi r_\delta(l_1) }{3\zeta_K(2)\sqrt{N(l_1)}}
\mathop{\text{Res}}_{\substack{ s=\pm \delta\\ s= \pm \tau}}
\biggl\{ \Gamma_{\delta}(s)
\left(\frac{2^{5/2}}{\pi}\right)^{2s}
\Big (\frac {X}{N(l_1)}\Big )^{s}\widehat{\Psi}(1+s)\frac{2s}{s^2 -\tau^2}
Z(1+2s;\delta)\eta_{\delta}(1+2s;l)
\biggr\}\\
&+
O\biggl(\frac{|r_\delta(l_1)|X^{\kappa+\epsilon}}{Y N(l_1)^{\frac 12+\kappa}}
+ \frac{|r_\delta(l_1)|X^{\epsilon}}{(XN(l_1))^{\frac 14}}
\biggr).
\end{split}
\end{align}

\subsection{The secondary main terms}

 We apply the Mellin transform to recast ${\mathcal R}_0(l)$ given in \eqref{PR} as
\begin{align}
\label{R0}
{\mathcal R}_{0}(l)
=& \frac{1}{2N(l)}\sum_{\substack{\alpha \equiv 1 \bmod (1+i)^3 \\  N(\alpha) \le Y\\ (\alpha,  l) =1}}
\frac{\mu_{[i]}(\alpha)}{N(\alpha)^2}\sum_{\substack{ k \in \mathcal O_K \\ k\neq 0 }}
\frac{(-1)^{N(k)}}{2\pi i}
\int\limits_{(c)}\sum_{\substack{n \equiv 1 \bmod (1+i)^3 \\ (n, \alpha)=1}} \frac{r_{\delta}(n)}
{ N(n)^{ \frac 32+w}} g(k,ln) h\Big ( \frac{N(k)X}{2N(\alpha^2l)}, w \Big )dw,
\end{align}
 where $c> |\Re(\delta)|$ and $h$ is defined as in \eqref{h}.

  We apply Lemma \ref{lem: nu-sum as an Euler product} in \eqref{R0} and move the line of integration to $\Re (w) = -\frac 12 + |\Re(\delta)| +\epsilon$. In this process,  we
encounter poles only when $k_1 =\pm i$, with simple poles at $w= \pm \delta$. The residues of these poles contribute secondary main terms.  We therefore write
${\mathcal R}_0(l) = {\mathcal R}(l) + {\mathcal P}_{+}(l)+{\mathcal P}_{-}(l)$ where
\begin{align}
\label{R}
\begin{split}
{\mathcal R}(l) =& \frac{1}{2N(l)}\sum_{\substack{\alpha \equiv 1 \bmod (1+i)^3 \\  N(\alpha) \le Y\\ (\alpha,  l) =1}}
\frac{\mu_{[i]}(\alpha)}{N(\alpha)^2}\sum_{\substack{ k \in \mathcal O_K \\ k\neq 0 }}
\frac{(-1)^{N(k)}}{2\pi i} \int\limits_{(-\frac 12 + |\Re(\delta)| +\epsilon)}
L(1+w+\delta,\chi_{i k_1}) L(1+w-\delta,\chi_{i k_1}) \\
& \times
{\mathcal G}_\delta(1+w;k,l,\alpha) h\Big ( \frac{N(k)X}{2N(\alpha^2l)}, w \Big ) dw,
\end{split}
\end{align}
and (after replacing $k$ by $\pm ik^2$)
\begin{align}
\begin{split}
\label{P}
{\mathcal P}_{+ }(l) =& \frac {\pi}{4}\cdot \frac{1}{2N(l)}\sum_{\substack{\alpha \equiv 1 \bmod (1+i)^3 \\  N(\alpha) \le Y\\ (\alpha,  l) =1}}
\frac{\mu_{[i]}(\alpha)}{N(\alpha)^2}
\sum_{\mu=\pm}  \zeta_K(1+2\mu \delta)
\sum_{k\in G} (-1)^{N(k)} {\mathcal G}_{\delta}(1+\mu\delta; i k^2,l,\alpha)
h\Big ( \frac{N(k)^2X}{2N(\alpha^2l)}, \mu \delta\Big ), \\
{\mathcal P}_{-}(l) =& \frac {\pi}{4} \cdot \frac{1}{2N(l)}\sum_{\substack{\alpha \equiv 1 \bmod (1+i)^3 \\  N(\alpha) \le Y\\ (\alpha,  l) =1}}
\frac{\mu_{[i]}(\alpha)}{N(\alpha)^2}
\sum_{\mu=\pm}  \zeta_K(1+2\mu \delta)
\sum_{k\in G} (-1)^{N(k)} {\mathcal G}_{\delta}(1+\mu\delta; -i k^2,l,\alpha)
h\Big ( \frac{N(k)^2X}{2N(\alpha^2l)}, \mu \delta\Big ).
\end{split}
\end{align}

  Note that it follows from Lemma \ref{Gausssum} and Lemma \ref{lem: nu-sum as an Euler product} that we have ${\mathcal G}_{\delta}(1+\mu\delta; i k^2,l,\alpha)={\mathcal G}_{\delta}(1+\mu\delta; -i k^2,l,\alpha)$. Note also that we have
\begin{align}
\label{Gsum}
\begin{split}
& \sum_{\substack{j \in G }}(-1)^{N(j)} N(j)^{-2u}  \mathcal{G}_{v}(s;  i j^2,\ell,\alpha)  \\
 = & -\sum_{\substack{j \in G \\ (j, 1+i) =1 }}N(j)^{-2u}  \mathcal{G}_{v}(s;  i j^2,\ell,\alpha )+\sum_{\substack{j \in G \\ 1+i|j }}N(j)^{-2u}  \mathcal{G}_{v}(s;  i j^2,\ell,\alpha ) \\
=& -\sum_{\substack{j \in G  }}N(j)^{-2u}  \mathcal{G}_{v}(s; i j^2,\ell,\alpha )+2\sum_{\substack{j \in G \\ 1+i|j }}N(j)^{-2u}  \mathcal{G}_{v}(s;  i j^2,\ell,\alpha ) \\
=& -\sum_{\substack{j \in G  }}N(j)^{-2u}  \mathcal{G}_{v}(s;  i j^2,\ell,\alpha )+2^{1-2u}\sum_{\substack{j \in G }}N(j)^{-2u}  \mathcal{G}_{v}(s;  i (1+i)^2j^2,\ell,\alpha ) \\
=& -(1-2^{1-2u})\sum_{\substack{j \in G  }}N(j)^{-2u}  \mathcal{G}_{v}(s; i j^2,\ell,\alpha ),
\end{split}
\end{align}
  where the last equality above follows from the observation that we have $\mathcal{G}_{v}(s; i (1+i)^2j^2,\ell,\alpha )=\mathcal{G}_{v}(s;  i j^2,\ell,\alpha )$ by Lemma \ref{Gausssum} and Lemma \ref{lem: nu-sum as an Euler product}.

 We now define for $\mu=\pm$ and any complex number $u$ with $\Re(u) >\frac 12$,
\begin{align*}
{\mathcal H}(u, v;l,\alpha) =& N(l)^u \sum_{k \in G}
\frac{(-1)^{N(k)}}{N(k)^{2u}} {\mathcal G}_{v}(1+v; i k^2,l,\alpha).
\end{align*}
 It follows from Lemma \ref{lem: nu-sum as an Euler product} that the above series converges absolutely when $\Re(u)>\frac 12$. We further apply \eqref{Gsum} to recast ${\mathcal H}(u, v;l,\alpha)$ as
\begin{align*}
{\mathcal H}(u, v;l,\alpha)
= -N(l)^u (1-2^{1-2u}) \prod_{\varpi \in G} \sum_{b=0}^{\infty} \frac{{\mathcal G}_{v;\varpi}
(1+v;i \varpi^{2b},l,\alpha)}{N(\varpi)^{2bu}}.
\end{align*}
  We now apply Lemma \ref{Gausssum} to evaluate ${\mathcal G}_{v; \varpi}$ defined in Lemma \ref{lem: nu-sum as an Euler product} to see that
\begin{align}
\label{Hexp}
{\mathcal H}(u,v;l,\alpha) = - N(l) (1-2^{1-2u}) N(l_1)^{u-\frac 12}
\zeta_K(2u) \zeta_K(2u+1+4v) {\mathcal H}_{1}(u, v;l,\alpha),
\end{align}
where ${\mathcal H}_1 =\prod_{\varpi \in G} {\mathcal H}_{1;\varpi}$ with
$$
{\mathcal H}_{1;\varpi} =
\begin{cases}
\Big(1-\frac 1{N(\varpi)} \Big) \Big(1-\frac{1}{N(\varpi)^{1+2v}}\Big)\Big(1-\frac{1}
{N(\varpi)^{2u+1+4v}}\Big) &\text{if } \varpi|2\alpha, \\
\Big(1-\frac{1}{N(\varpi)}\Big) \Big(1-\frac{1}{N(\varpi)^{1+2v}}\Big)
\Big(1+\frac 1{N(\varpi)} +\frac 1{N(\varpi)^{1+2v}} -\frac{1}{N(\varpi)^{2u+2+4v}}
\Big) &\text{if } \varpi \nmid 2\alpha l, \\
\Big(1-\frac{1}{N(\varpi)}\Big) \Big(1-\frac{1}{N(\varpi)^{1+2v}}\Big)
\Big(1+\frac{1}{N(\varpi)^{2u+2v}}\Big) &\text{if } \varpi|l_1, \\
\Big(1-\frac{1}{N(\varpi)}\Big) \Big(1-\frac{1}{N(\varpi)^{1+2v}}\Big)
\Big(1+\frac{1}{N(\varpi)^{1+2v}}\Big) &\text{if } \varpi |l, \varpi \nmid l_1.
\end{cases}
$$
 It then follows from this and \eqref{Hexp} that ${\mathcal H}(u, v;l,\alpha)$
is analytic in the domain $\Re (u) > -\frac 12 + |\Re(v)|$.

We now apply Lemma \ref{lem: properties of h(xi,w)} and the observation that ${\mathcal G}_{v}(1+v; i k^2,l,\alpha)$ is an even function of $v$ to  deduce that, for $c> \max(\frac 12 +|\Re(\delta)|, |\Re(\tau)|)= \frac 12+ |
\Re(\delta)|$,
\begin{align}
\begin{split}
\label{Hint}
& \sum_{k \in G} (-1)^{N(k)} {\mathcal G}_{\delta}(1+\mu\delta; \pm i k^2,l,\alpha)
h\Big ( \frac{N(k)^2X}{2N(\alpha^2l)}, \mu \delta\Big ) \\
=& \frac{\pi}{2\pi i} \int\limits_{(c)} \Gamma_{\delta}(s)
\leg{2^6N(\alpha)^2}{\pi^2}^s
{\widehat \Psi}(1+\mu\delta) \leg{X}{2N(\alpha)^2}^{\mu\delta}
\frac{\Gamma (s-\mu \delta)}{\Gamma (1-s+\mu \delta)} \pi^{-2s+2\mu\delta}\frac{2s}{s^2-\tau^2}
{\mathcal H}(s-\mu\delta, \mu\delta; l,\alpha) ds.
\end{split}
\end{align}

  The analytical properties of $H$ discussed above allow us to move the line of integration in \eqref{Hint} to $\Re (s) = \kappa+ \frac{1}{\log X}$ without encountering any pole. We then deduce from this and
\eqref{P} that
\begin{align}
\label{Pexpression}
\begin{split}
{\mathcal P}_{\pm }(l)  =& \frac {\pi}{4} \cdot \frac{\pi }{2\pi i} \int\limits_{(\kappa+\frac{1}{\log X})}
\Gamma_\delta(s) \leg{2^6}{\pi^2}^s \frac{2s}{s^2-\tau^2}
 \sum_{\mu=\pm} {\widehat \Psi}(1+\mu\delta)
\leg{X}{2}^{\mu\delta} \zeta_K(1+2\mu\delta)
 \frac{\Gamma (s-\mu \delta)}{\Gamma (1-s+\mu \delta)} \pi^{-2s+2\mu\delta} \\
&\times
 \frac{1}{2N(l)}\sum_{\substack{\alpha \equiv 1 \bmod (1+i)^3 \\  N(\alpha) \le Y\\ (\alpha,  l) =1}}
\frac{\mu_{[i]}(\alpha)}{N(\alpha)^{2-2s+2\mu\delta}}
{\mathcal H}(s-\mu\delta, \mu \delta;l,\alpha) ds.
\end{split}
\end{align}

  We now extend the sum over $\alpha$ in \eqref{Pexpression} to infinity. To estimate the error introduced, we let ${\mathcal C}$ be the circle centred at $0$, with radius $\kappa+ 1/(2 \log X)$, and oriented
counter-clockwise.
Notice that, for any complex number $s$ with $\Re(s) =\kappa +\frac{1}{\log X}$, the
 function $2z {\widehat \Psi}(1+z) \leg{X}{2N(\alpha)^2}^z \zeta_K(1+2z) \frac{\Gamma (s-\mu \delta)}{\Gamma (1-s+\mu \delta)} (\pi)^{-2s+2\mu\delta} {\mathcal H}(s-z, z;l,\alpha)$ is analytic for $z$ inside
 ${\mathcal C}$. Thus, we deduce from Cauchy's theorem that
\begin{align}
\label{summubound}
\begin{split}
&\sum_{\mu=\pm} {\widehat \Psi}(1+\mu\delta)
\leg{X}{2N(\alpha)^2}^{\mu\delta} \zeta_K(1+2\mu\delta)
 \frac{\Gamma (s-\mu \delta)}{\Gamma (1-s+\mu \delta)} \pi^{-2s+2\mu\delta}
{\mathcal H}(s-\mu\delta, \mu \delta;l,\alpha)\\
=& \frac{1}{2\pi i} \int_{{\mathcal C}}
{\widehat \Psi}(1+z) \leg{X}{2N(\alpha)^2}^z \zeta_K(1+2z) \frac{\Gamma (s-z)}{\Gamma (1-s+z)} \pi^{-2s+2z} {\mathcal H}(s-z, z;l,\alpha) \frac{2z}{z^2 -\delta^2} dz.
\end{split}
\end{align}

 We observe that $2\kappa + 3/(2\log X)
 \ge \Re(s-z) \ge 1/(2\log X)$ for $z$ on ${\mathcal C}$ and that
$|\frac{2z}{z^2-\delta^2}{\widehat \Psi}(1+z) \zeta_K(1+2z)
\leg{X}{2N(\alpha)^2}^z|
\ll (\log^2 X) (XN(\alpha)^2)^{\kappa+1/\log X}$. Also, by Stirling's
formula (see \cite[(5.112)]{iwakow}), we have
$$\Big|\frac{\Gamma (s-z)}{\Gamma (1-s+z)} \pi^{-2s+2z} \Big | \ll (1+|\Im(s)|)^{4\kappa + 3/\log X-1} \ll 1.$$
Moreover, applying the convexity bounds for $\zeta_K(s)$ given in \eqref{zetaconvexitybounds}, we deduce from \eqref{Hexp} that $|{\mathcal H}(s-z, z;l,\alpha) |
\ll N(l)^{1+\epsilon} N(l_1)^{2\kappa + 3/(2\log X)-\frac 12} N(\alpha)^{\epsilon} (1+|s|) \log X$.
These estimates allow us to bound the quantities in \eqref{summubound} by
$$
 \ll N(l)^{1+\epsilon} N(l_1)^{2\kappa-\frac 12}
 (XN(\alpha)^2)^{\kappa+\epsilon} (1+|s|).
$$
It follows that
\begin{align*}
& \frac{1}{2N(l)} \sum_{\substack{ \alpha \equiv 1 \bmod (1+i)^3 \\ N(\alpha)>Y \\ (\alpha,l)=1}} \frac{\mu_{[i]}(\alpha)}
{N(\alpha)^{2-2s}}
\sum_{\mu=\pm} {\widehat \Psi}(1+\mu\delta)
\leg{X}{2N(\alpha)^2}^{\mu\delta} \zeta_K(1+2\mu\delta)
 \frac{\Gamma (s-\mu\delta)}{\Gamma (1-s+\mu\delta)} \pi^{-2s+2z}
{\mathcal H}(s-\mu\delta, \mu \delta;l,\alpha) \\
\ll & N(l)^{\epsilon} N(l_1)^{2\kappa
-\frac 12} X^{\kappa+\epsilon}
(1+|s|)
Y^{-1 +4 \kappa }.
\end{align*}
Applying the above in \eqref{Pexpression}, we see that the error introduced by
extending the sum over $\alpha$ to infinity is
\begin{align*}
&\ll N(l)^{\epsilon} N(l_1)^{2\kappa
-\frac 12} X^{\kappa+\epsilon}
Y^{-1+4\kappa}
\int\limits_{(\kappa+\frac 1{\log X})} |\Gamma_\delta(s)| (1+|s|)
\frac{|s|}{|s^2-\tau^2|}|ds| \ll N(l)^{\epsilon} N(l_1)^{2\kappa
-\frac 12} X^{\kappa+\epsilon}
Y^{-1+4\kappa}.
\end{align*}
  We then conclude that we have
\begin{align}
\label{PJ}
 {\mathcal P}_{\pm}(l)=\frac {\pi^2}{4} \cdot \sum_{\mu=\pm} {\widehat \Psi}(1+\mu\delta)
\leg{X}{2}^{\mu\delta} \zeta_K(1+2\mu\delta)
\frac{1}{2\pi i} \int\limits_{(\kappa+\frac 1{\log X})}
{\mathcal J}(s, \mu \delta;l) \frac{2s}{s^2-\tau^2} ds+O(N(l)^{\epsilon} N(l_1)^{2\kappa
-\frac 12} X^{\kappa+
\epsilon}Y^{-1+4\kappa}),
\end{align}
where
\begin{align*}
{\mathcal J}(s, v;l) :=& \Gamma_{v}(s) \leg{2^6}{\pi^2}^s \frac{\Gamma (s-v)}{\Gamma (1-s+v)} \pi^{-2s+2z} {\mathcal K}(s,v;l), \\
{\mathcal K}(s, v;l) :=& \frac{1}{2N(l)}\sum_{\substack{\alpha \equiv 1 \bmod (1+i)^3 \\   (\alpha,  l) =1}}
\frac{\mu_{[i]}(\alpha)}{N(\alpha)^{2-2s+2\mu\delta}}
{\mathcal H}(s-v, v;l,\alpha) .
\end{align*}

 A direct calculation using the expression for ${\mathcal H}$ given in \eqref{Hexp} yields
\begin{align*}
{\mathcal K}(s, v;l) =& -\frac 1{4N(l_1)^{\frac {1}{2}+v}} \frac {\phi_{[i]}(l)}{N(l)}
\prod_{\substack{\varpi \in G \\ \varpi|2l}} \biggl( 1-\frac{1}{N(\varpi)^{1+2v}} \biggr)
\prod_{\substack{\varpi \in G \\ \varpi |l\\ \varpi \nmid l_1}} \biggl( 1+ \frac{1}{N(\varpi)^{1+2v}}\biggr)
r_{s}(l_1) \\
&\times
\biggl(\frac{4^s+4^{-s}-2^{-1-2v}-2^{1+2v}}{4^s}\biggr)
\zeta_K(2s-2v) \zeta_K(2s+1+2v) \\
&\times
\prod_{\substack{\varpi \in G \\ \varpi \nmid 2l}} \biggl (1-\frac 1{N(\varpi)}\biggr)\biggl( 1-
\frac{1}{N(\varpi)^{1+2v}}\biggr)\biggl(1+\frac 1{N(\varpi)} +\frac1{N(\varpi)^{1+2v}}
+ \frac{1}{N(\varpi)^{3+4v}}-\frac{(N(\varpi)^{2s}+N(\varpi)^{-2s})}{N(\varpi)^{2+2v}}
\biggr).
\end{align*}

Using this together with the
functional equation for $\zeta_K(s)$ given in \eqref{fcneqnforzeta}, and the identity $\Gamma(z) \Gamma(z+\frac 12)= \pi^{\frac12} 2^{1-2z} \Gamma(2z)$,  we obtain the following identity
\begin{align}
\label{Jindentity}
\zeta_K(1+2v) {\mathcal J}(s,v;l)
=& \frac 1{\pi} \cdot \frac{4r_{s}(l_1)}{3\zeta_K(2)\sqrt{N(l_1)}}\leg{2^6}{\pi^2 N(l_1)}^{v}
\Gamma_s(v) Z(1+2v;s) \eta_s(1+2v;l).
\end{align}
 As it is easy to see that the left side above is invariant under $s\to -s$, we deduce that
${\mathcal J}(s,v;l) ={\mathcal J}(-s,v;l)$.

 We now evaluate the integral in \eqref{PJ} for $\mu =\pm$ by moving the line of integration to $\Re(s) = - \kappa-\frac 1{\log X}
$ to encounter simple poles at $s=\pm \delta, \pm
\tau$ in the process.  It follows that
\begin{align*}
\frac{1}{2\pi i} \int\limits_{(\kappa+\frac 1{\log X})}
{\mathcal J}(s, \mu \delta;l) \frac{2s}{s^2-\tau^2} ds=&  \mathop{\text{Res}}_{s=\pm\delta, \pm\tau} {\mathcal J}(s, \mu \delta;l)
\frac{2s}{s^2-\tau^2}
+ \frac{1}{2\pi i} \int\limits_{(-\kappa-\frac 1{\log X})}
{\mathcal J}(s, \mu \delta;l) \frac{2s}{s^2-\tau^2} ds \\
=& \mathop {\text{Res}}_{s=\pm\delta, \pm\tau}{\mathcal J}(s, \mu \delta;l)
\frac{2s}{s^2-\tau^2}
- \frac{1}{2\pi i} \int\limits_{(\kappa+\frac 1{\log X})}
{\mathcal J}(s, \mu \delta;l) \frac{2s}{s^2-\tau^2} ds,
\end{align*}
 where the last integral above follows from the previous one upon a change of variable $s \rightarrow -s$ while noticing that ${\mathcal J}(s, \mu \delta;l)
= {\mathcal J}(-s, \mu \delta;l)$. Thus we deduce that
$$
\frac{1}{2\pi i} \int\limits_{(\kappa+\frac 1{\log X})}
{\mathcal J}(s, \mu \delta;l) \frac{2s}{s^2-\tau^2} ds=\frac{1}{2} \biggl( \mathop{\text{Res}}_{s=\pm\delta, \pm \tau}
{\mathcal J}(s, \mu \delta;l)
\frac{2s}{s^2-\tau^2} \biggr)
= \mathop{\text{Res}}_{s=\mu\delta} {\mathcal J}(s, \mu \delta;l)
\frac{2s}{s^2-\tau^2} + {\mathcal J}(\tau, \mu \delta;l).
$$

 Applying the above in \eqref{PJ}, we obtain that
\begin{align}
\label{P+-}
\begin{split}
{\mathcal P}_{+}(l) ={\mathcal P}_{-}(l) =& \frac {\pi^2}{4} \cdot \sum_{\mu=\pm } {\widehat \Psi}(1+\mu\delta)
\leg{X}{2}^{\mu\delta} \zeta_K(1+2\mu\delta) \left(
 \mathop{\text{Res}}_{s=\mu\delta} {\mathcal J}(s, \mu \delta;l)
\frac{2s}{s^2-\tau^2} + {\mathcal J}(\tau, \mu \delta;l)
\right) \\
&
+ O\Big(\frac{N(l)^\epsilon X^{\kappa +\epsilon}
 N(l_1)^{2\kappa -\frac 12}}{Y^{1-4\kappa}}
\Big).
\end{split}
\end{align}

\subsection{Estimation of
${\mathcal R}(l)$}

 In this section, we estimate ${\mathcal R}(l)$ given in \eqref{R} on average by deriving the bound given in \eqref{Rbound}.  We denote $\beta_l = \frac{\overline{{\mathcal R}(l)}}{|{\mathcal R}(l)|}$ if ${\mathcal R}(l)
\neq 0$, and $\beta_l =1$ otherwise. We deduce from \eqref{R} that
\begin{align*}
\sum_{L \leq N(l) \leq 2L-1} |{\mathcal R}(l)| =\sum_{L \leq N(l) \leq 2L-1}\beta_l {\mathcal R}(l) \ll& \sum_{\substack{\alpha \equiv 1 \bmod (1+i)^3 \\  N(\alpha) \le Y}} \frac{1}{N(\alpha)^2}
 \int\limits_{(-\frac 12+|\Re(
\delta)|+\epsilon)}
\sum_{\substack{ k \in \mathcal O_K \\ k\neq 0 }} U_{\delta}(\alpha,k,w)
|dw|,
\end{align*}
  where
\begin{equation*}
U_{\delta}(\alpha,k,w) = |L(1+w+\delta,\chi_{ik_1})L(1+w-\delta,\chi_{ik_1})| \biggl|\sum_{\substack{L \leq N(l) \leq 2L-1 \\(l,2\alpha)=1}}
\frac{\beta_l}{N(l)}
{\mathcal G}_\delta(1+w;k,l,\alpha) h\biggl(\frac{N(k)X}{2N(\alpha^2l)},w\biggr)\biggr|.
\end{equation*}

  Recall the definition of $k_2$ in \eqref{eq: defn of k1 and k2}. We apply the Cauchy-Schwarz inequality to see that for any integer $K \geq 1$,
\begin{align*}
\begin{split}
  \sum_{\substack{ K\leq N(k)< 2K }}U_{\delta}(\alpha,k,w)
\ll & \Big(\sum_{N(k)=K}^{2K-1} N(k_2)
|L(1+w+\delta,\chi_{ik_1})L(1+w-\delta,\chi_{ik_1})|^2
\Big)^{\frac 12} \\
& \times \biggl(\sum_{N(k)=K}^{2K-1} \frac{1}{N(k_2)}
\biggl|\sum_{\substack{ N(l)=L\\
(l,2\alpha)=1}}^{2L-1}
\frac{\beta_l}{N(l)} {\mathcal G}_\delta(1+w;k,l,\alpha)
h\biggl(\frac{N(k)X}{2N(\alpha^2 l)},w
\biggr)\biggr|^2\ \biggr)^{\frac 12}.
\end{split}
\end{align*}

 Using the Cauchy-Schwarz inequality again together with Lemma \ref{lem:2.3}, we see that
$$
\Big(\sum_{N(k)=K}^{2K-1} N(k_2)
|L(1+w+\delta,\chi_{ik_1})L(1+w-\delta,\chi_{ik_1})|^2
\Big)^{\frac 12}
\ll (K(1+|w|^2))^{\frac 12+\epsilon}.
$$

  This implies that
\begin{align}\label{eq:using moment estimates of Heath Brown}
\begin{split}
& \sum_{\substack{ K\leq N(k)<2K }}U(\alpha,k,w)
\ll  (K(1+|w|^2))^{\frac 12+\epsilon} \biggl(\sum_{N(k)=K}^{2K-1} \frac{1}{N(k_2)}
\biggl|\sum_{\substack{ N(l)=L\\
(l,2\alpha)=1}}^{2L-1}
\frac{\beta_l}{N(l)} {\mathcal G}_\delta(1+w;k,l,\alpha)
h\biggl(\frac{N(k)X}{2N(\alpha^2 l)},w
\biggr)\biggr|^2\ \biggr)^{\frac 12}.
\end{split}
\end{align}

   Now, applying the first bound of Lemma~\ref{lem: version of Lemma 5.4 of Sound} as well as the estimation given in \eqref{gsest} for $ \Psi(1+w)$, we see that we may restrict the sum of the right side of \eqref{eq:using moment estimates of Heath Brown} to $K=2^j \leq N(\alpha)^2 L(1+|w|^2)(\log X)^4$. We then apply the second bound in Lemma~\ref{lem: version of Lemma 5.4 of Sound} to \eqref{eq:using moment estimates of Heath Brown} to see that
\begin{align*}
\begin{split}
 \sum_{\substack{ K\leq N(k)<2K }}U(\alpha,k,w)
 \ll &  (K(1+|w|^2))^{\frac 12+\epsilon} \times ((1+|w|)N(\alpha) LX)^{\varepsilon}|\widehat{\Phi}(1+w)|  \Big(\frac{N(\alpha)^{2} L(1+|w|^2)}{K}\Big)^{|\Re(\tau)|-|\Re(
\delta)|}  \\
& \times \frac{N(\alpha)L^{1/2}}
{K^{1/2} X^{1/2-|\Re(\delta)|}}(K^{1/2}+L^{1/2})\\
 \ll &  (1+|w|^2)^{\frac 12+\epsilon} \times (N(\alpha) LX)^{\varepsilon}|\widehat{\Phi}(1+w)|  \Big(\frac{N(\alpha)^{2} L(1+|w|^2)}{K}\Big)^{|\Re(\tau)|-|\Re(
\delta)|}  \\
& \times \frac{N(\alpha)L^{1/2}}
{X^{1/2-|\Re(\delta)|}}(K^{1/2}+L^{1/2}).
\end{split}
\end{align*}
 Summing over $K=2^j, K\leq N(\alpha)^2 L(1+|w|^2)(\log X)^4$, we deduce from the above that
\begin{equation*}
\sum_{\substack{ k \in \mathcal O_K \\ k \neq 0 }}U(\alpha,k,w) \ll  (1+|w|^2)^{1+\epsilon}(N(\alpha) L)^{\varepsilon}|\widehat{\Phi}(1+w)|\Big ( \frac {N(\alpha)^{2} L}{X^{\frac{1}{2}-|\Re(\delta)|-\varepsilon}}+ \frac {N(\alpha)^{1+2\kappa} L^{1+\kappa}}{X^{\frac{1}{2}-|\Re(\delta)|-\varepsilon}} \Big ).
\end{equation*}

  It follows that
\begin{align*}
\sum_{l=L}^{2L-1} |{\mathcal R}(l)| \ll &
\frac{L^{1+\epsilon} }{X^{\frac 12-|\Re(\delta)| -\epsilon}}
\sum_{N(\alpha) \le Y}
N(\alpha)^{\epsilon}(1+N(\alpha)^{-1+2\kappa}L^{\kappa}) \int\limits_{(-\frac 12+|\Re(\delta)|+\epsilon)} |{\widehat \Psi}(1+w)|
(1+|w|^2)^{1+\epsilon} |dw|.
\end{align*}
 Now, the bound given in \eqref{Rbound} follows from this by using \eqref{gsest} for $\Psi(1+w)$ with $n= 3$ when $|w| \leq \Psi_{(4)}/\Psi_{(3)}$ and $n = 4$ otherwise to evaluate the integral above.

\subsection{Conclusion}

 Combining \eqref{SMdecomp}, \eqref{Pl}, \eqref{P+-} and our result for the average size of ${\mathcal R}(l)$ given in the previous section, we see that in order to prove Proposition \ref{Proposition 2.3}, it remains to simply the expression ${\mathcal P}(l) + {\mathcal P}_{+}(l)+ {\mathcal P}_{-}(l)$ given in \eqref{Pl} and \eqref{P+-}. We now employ the identity given in \eqref{Jindentity} to see that the contribution from the poles at $\mu\delta$ to \eqref{Pl} and \eqref{P+-} cancel precisely each other, while the contribution from the
poles at $\pm \tau$ in both these expressions gives rise to the  main term on the right side of \eqref{SMasymp},
This completes the proof of Proposition \ref{Proposition 2.3}.

\section{Proof of Theorem \ref{thm: nonvanishing}}

  We now proceed to complete our proof of Theorem \ref{thm: nonvanishing}. We begin by evaluating ${\mathcal W}(\delta_1, \Phi)$ when $\delta_1$ is near $0$.

\subsection{Proof of Proposition \ref{Proposition 2.4}}

  We apply Lemma \ref{Lemma 3.2} together with the definition of $\xi(s,\chi_{(1+i)^5d})$ given in \eqref{xidef} to see that, by setting $\Phi_{-\tau}(t) = t^{-\tau} \Phi(t)$,
\begin{align}
\label{SW1}
& {\mathcal S}(1; \Phi) {\mathcal W}(\delta_1, \Phi)
= \frac{(2^5 X/\pi^2 )^{-\tau}}{\Gamma_\delta(\tau)}
{\mathcal S}(A_{\delta,\tau}(d)|M(\tfrac 12+\delta_1,d)|^2; \Phi_{-\tau}).
\end{align}

 We further set $Y=X^{4\kappa}$ to rewrite the above as
\begin{align}
\label{SW2}
{\mathcal S}(1; \Phi) {\mathcal W}(\delta_1, \Phi)=& \frac{(2^5 X/\pi^2 )^{-\tau}}{\Gamma_\delta(\tau)}
\Big\{
{\mathcal S}_M
(A_{\delta,\tau}(d)|M(\tfrac 12+\delta_1,d)|^2; \Phi_{-\tau})
+O({\mathcal S}_R
(A_{\delta,\tau}(d)|M(\tfrac 12+\delta_1,d)|^2; \Phi_{-\tau}))
\Big\} \\
=& \frac{(2^5 X/\pi^2 )^{-\tau}}{\Gamma_\delta(\tau)}
{\mathcal S}_M
(A_{\delta,\tau}(d)|M(\tfrac 12+\delta_1,d)|^2; \Phi_{-\tau})
+ O\Big(X^{-\kappa+\varepsilon}\Big),
\end{align}
 where the last estimation above follows from Proposition \ref{Proposition 2.2}.

Note that we have
\begin{align}
\label{SW3}
 & {\mathcal S}_M(A_{\delta,\tau}(d)|M(\tfrac 12+\delta_1,d)|^2; \Phi_{-\tau}) \\
= &
\sum_{\substack{l \equiv 1 \bmod (1+i)^3}} \Big(\sum_{\substack{r,s \equiv 1 \bmod (1+i)^3 \\ rs=l}}
\frac{\lam(r) \lam(s)}{N(r)^{\frac 12+\delta_1}N(s)^{\frac 12+\delta_2} }
\Big) {\mathcal S}_M \Big(A_{\delta,\tau}(d)\leg{(1+i)^5 d}{l};\Phi_{-\tau}\Big),
\end{align}
 We now apply Proposition \ref{Proposition 2.3} to evaluate
${\mathcal S}_M(A_{\delta,\tau}(d)|M(\tfrac 12+\delta_1,d)|^2; \Phi_{-\tau})$. Using the estimations that $|\lam(n)|\ll N(n)^\epsilon$, $r_\delta(n)\ll N(n)^{\epsilon}$, $d_{[i]}(n) \ll N(n)^{\epsilon}$, and $\tau \ge -1/\varepsilon \log X$, we deduce that
various remainder terms in Proposition \ref{Proposition 2.3} contribute
\begin{align}
\label{SW4}
&\ll \sum_{N(l) \le M^2} \frac{N(l)^{\varepsilon}}{N(l)^{\frac 12+\tau}}
\Big( \frac{X^{\varepsilon}}{(XN(l_1))^{\frac 14}}
+\frac{N(l)^\varepsilon X^{\kappa +\varepsilon}
 N(l_1)^{2\kappa-\frac 12}}{Y^{1-4\kappa}} + |{\mathcal R}(l)|\Big) \ll X^{-\kappa +\varepsilon} \Phi_{(3)} \Phi_{(4)}^{\epsilon}.
\end{align}

We obtain from \eqref{SW1}-\eqref{SW4} that
\begin{align}
\label{WintermofM}
{\mathcal S}(1; \Phi){\mathcal W}(\delta_1, \Phi)
= \frac{(2^5X/\pi^2)^{-\tau}}{\Gamma_\delta(\tau)}
\sum_{\substack{l \equiv 1 \bmod (1+i)^3}} \Big(\sum_{\substack{r,s \equiv 1 \bmod (1+i)^3 \\ rs=l}}
\frac{\lam(r) \lam(s)}{N(r)^{\frac 12+\delta_1}N(s)^{\frac 12+\delta_2} }
\Big) {\mathcal M}(l)
 + O(X^{-\kappa -\epsilon} \Phi_{(3)} \Phi_{(4)}^{\epsilon}),
\end{align}
where ${\mathcal M}(l) = {\mathcal M}_1(l)+{\mathcal M}_2(l)$ with
\begin{align*}
{\mathcal M}_1(l) =& \frac{2\pi  }{3\zeta_K(2)\sqrt{N(l_1)}}
\sum_{\mu=\pm}r_\delta(l_1)\Gamma_\delta(\mu\tau)
\leg {2^5X}{N(l_1) \pi^2}^{\mu\tau} {\widehat \Phi}(1+\mu\tau-\tau) Z(1+2\mu\tau;\delta)
\eta_{\delta}(1+2\mu\tau;l), \\
{\mathcal M}_2(l)
=&  \frac{2\pi  }{3\zeta_K(2)\sqrt{N(l_1)}}
\sum_{\mu=\pm} r_{\tau}(l_1) \Gamma_\tau(\mu\delta) \leg{2^5X}{N(l_1) \pi^2}^{\mu\delta}
{\widehat \Phi}(1+\mu\delta-\tau) Z(1+2\mu\delta;\tau) \eta_{\tau}(1+2\mu\delta;l).
\end{align*}

 By taking ${\mathcal C}$ to be the closed contour described in the paragraph above Lemma \ref{Lemma 3.17}, we can evaluate the contribution of ${\mathcal M}_1(l)$ to \eqref{WintermofM} by
\begin{align}
\label{M1innersum}
\begin{split}
&\frac{1}{2\pi i} \int\limits_{{\mathcal C}}
\frac{2\pi {\widehat \Phi}(1+w-\tau)}{3\zeta_K(2)\Gamma_\delta(\tau)}
\leg{2^5X}{\pi^2}^{w-\tau}
Z(1+2w;\delta) \Gamma_\delta(w) \frac{2w}{w^2-\tau^2}
\\
& \times \biggl\{ \sum_{\substack{l \equiv 1 \bmod (1+i)^3}}\frac{r_\delta(l_1)}{N(l_1)^{1/2+w}}
 \Big(\sum_{\substack{r,s \equiv 1 \bmod (1+i)^3 \\ rs=l}}
\frac{\lam(r) \lam(s)}{N(r)^{\frac 12+\delta_1}N(s)^{\frac 12+\delta_2} } \Big ) \eta_{\delta}(1+2w;l)\biggr\} dw.
\end{split}
\end{align}

Using the fact that $\lam$ is supported on square-free elements in $\mathcal O_K$, we now write $r=\alpha a$, $s=\alpha b$ with $\alpha$, $a$, $b$ being primary, square-free and $(a,b)=1$. This implies that $l=\alpha^2 a b$, $l_1=ab$, and $l_2 =\alpha$ so that the sum over $l$ in \eqref{M1innersum} can be written as
\begin{align}
\label{M1sumoverl}
\begin{split}
&\sum_{\alpha \equiv 1 \bmod (1+i)^3 } \sum_{\substack{ a, b \equiv 1 \bmod (1+i)^3 \\ (a,b)=1}}
\frac{r_{\delta}(ab)}{N(ab)^{\frac 12 +w}}
\frac{\lam(\alpha a)\lam(\alpha b)}{N(\alpha)^{1+\delta_1+\delta_2}
N(a)^{\frac 12+\delta_1} N(b)^{\frac 12+\delta_2} }
\eta_{\delta}(1+2w;\alpha^2 ab)\\
=& \sum_{\alpha \equiv 1 \bmod (1+i)^3 } \frac{1}{N(\alpha)^{1+2\tau}}
\sum_{\substack{ a, b \equiv 1 \bmod (1+i)^3 \\ (a,b)=1}} \frac{r_{\delta}(a) r_{\delta}(b)}
{N(a)^{1+\delta_1+w}N(b)^{1+\delta_2+w}} \lam(\alpha a) \lam(\alpha b)
\eta_{\delta}(1+2w;\alpha^2 ab).
\end{split}
\end{align}

 We then apply the relation
$$
\eta_{\delta}(1+2w;\alpha^2 ab) =
\frac{\eta_{\delta}(1+2w;1)}{h_{w}(\alpha) h_{w}(a) h_{w}(b)}
\prod_{\substack{\varpi \in G\\ \varpi|\alpha}} \Big( 1+\frac{1}{N(\varpi)^{1+2w}}\Big)
$$
to recast the expressions given in \eqref{M1sumoverl} as
$$
\eta_{\delta}(1+2w;1) \sum_{\alpha \equiv 1 \bmod (1+i)^3 }
\frac{1}{N(\alpha)^{1+2\tau} h_{w}(\alpha)} \prod_{\substack{\varpi \in G \\ \varpi|\alpha}}\Big(1
+\frac{1}{N(\varpi)^{1+2w}}\Big) \sum_{\substack{  a,b \equiv 1 \bmod (1+i)^3\\ (a,b)=1}}
\frac{r_{\delta}(a) \lam(\alpha a)}{N(a)^{1+\delta_1+w}h_{w}(a)}
\frac{r_{\delta}(b)\lam(\alpha b)}{N(b)^{1+\delta_2+w}h_{w}(b)}.
$$
 Next, we use the M\"obius function to detect the condition $(a,b)=1$ in the above expression to see that it equals to
\begin{align}
\label{M1sumoverlinab}
\begin{split}
& \eta_{\delta}(1+2w;1)\sum_{\alpha \equiv 1 \bmod (1+i)^3 }
\frac{1}{N(\alpha)^{1+2\tau} h_{w}(\alpha)} \prod_{\substack{\varpi \in G \\ \varpi|\alpha}}\Big(1
+\frac{1}{N(\varpi)^{1+2w}}\Big)
\sum_{\substack{ \beta \equiv 1 \bmod (1+i)^3 }} \frac{r_{\delta}(\beta)^2 \mu_{[i]}(\beta) }{N(\beta)^{2+2\tau
+2w} h_{w}(\beta)^2 } \\
&\times \sum_{\substack{  a,b \equiv 1 \bmod (1+i)^3}}
\frac{r_{\delta}(a) \lam(a\alpha \beta) }{N(a)^{1+\delta_1+w} h_{w}(a)}
\frac{r_{\delta}(b)\lam(b\alpha\beta)}{N(b)^{1+\delta_2+w}h_{w}(b)}.
\end{split}
\end{align}

 We further group terms according to $\gamma =\alpha \beta$ to see
that \eqref{M1sumoverlinab} becomes
\begin{align}
\label{M1sumoverlinHab}
\begin{split}
\eta_{\delta}(1+2w;1) \sum_{\gamma \equiv 1 \bmod (1+i)^3} \frac{H_{w}(\gamma)}
{N(\gamma)^{1+2\tau} h_{w}(\gamma)} \Big( \sum_{\substack{  a \equiv 1 \bmod (1+i)^3}} \frac{r_{\delta}(a)
\lam(a\gamma)}{N(a)^{1+\delta_1+w}h_{w}(a)}\Big)
\Big( \sum_{\substack{  b \equiv 1 \bmod (1+i)^3}}\frac{r_{\delta}(b) \lam(b\gamma)}{N(b)^{1+\delta_2+w}
h_{w}(b)}\Big).
\end{split}
\end{align}

We now apply Lemma \ref{Lemma 3.16} to evaluate
the sum over $a$ given in \eqref{M1sumoverlinHab} by first considering the case $N(\gamma) \le M^{1-b}$.
We set $u=\delta_1+w$, $v=\delta$, $g(n)=1/h_w(n)$ and we denote $G(s,\gamma;u,v)$ by $G_w(s,\gamma;u,v)$ in Lemma \ref{Lemma 3.16} to see that by applying it with $y=M$, $R(x)=P(x)$ and then with $y=M^{1-b}$, $R(x)= (1-P(b+x(1-b)))$, the sum of these two applications yields
\begin{align}
\label{sumovera}
\begin{split}
& \sum_{\substack{ N(a) \le M/N(\gamma) \\ a \equiv 1 \bmod (1+i)^3 }}
\frac{r_{\delta}(a) \mu_{[i]}(a\gamma)}{N(a)^{1+\delta_1+w }
h_w(a)} Q\Big(\frac{\log(M/N(a\gamma))}{\log M}\Big) \\
=& \frac{\mu_{[i]}(\gamma)G_w(1,\gamma;\delta_1+w,\delta)}{\zeta_K(1+\delta_1+w+\delta)
\zeta_K(1+\delta_1+w-\delta)}
  + O\Big( \frac{E(\gamma)}{\log^2 M} \leg{M^{1-b}}
{N(\gamma)}^{-\tau-\Re(w)} \exp(-A_0 \sqrt{\log M^{1-b}/N(\gamma)}\Big).
\end{split}
\end{align}
  Here one checks that the sum of the $k \geq 1$ terms of the above two applications of Lemma \ref{Lemma 3.16} cancel each other.
Thus the main term above comes from the sum of the $k=0$ terms only. Note also that the main term above is $\ll |\delta_1|^2$ due to the choice of the contour ${\mathcal C}$.

  Now, replacing $a$ by $b$ and $\delta_1$ by $\delta_2$ in \eqref{sumovera}, we see that a similar expression for the sum over $b$ in \eqref{M1sumoverlinHab} holds. It follows that the case $N(\gamma) \le M^{1-b}$ in
\eqref{M1sumoverlinHab} equals
\begin{align*}
\begin{split}
& \eta_\delta(1+2w;1) \sum_{\substack{ N(\gamma) \le M^{1-b} \\ \gamma \equiv 1 \bmod (1+i)^3
}}
\frac{H_w(\gamma)}{N(\gamma)^{1+2\tau}h_w(\gamma)} \frac{\mu^2_{[i]}(\gamma)
G_w(1,\gamma;\delta_1+w,\delta)G_w(1,\gamma;\delta_2+w,\delta)}
{\prod_{\mu=\pm} \zeta_K(1+\delta_1+w+\mu\delta)\zeta_K(1+\delta_2+w+\mu\delta)}
\\
& + O\Big( \sum_{\substack{ N(\gamma) \le M^{1-b} \\ \gamma \equiv 1 \bmod (1+i)^3
}} \frac{H_w(\gamma)}{N(\gamma)^{1+2\tau}h_w(\gamma)} E(\gamma)\frac{|\delta_1|^2 }{\log^2 M}
\leg{M^{1-b}}{N(\gamma)}^{-\tau-\Re(w)} \exp(-A_0 \sqrt{\log (M^{1-b}/N(\gamma))}
\Big) \\
=& \eta_\delta(1+2w;1)\sum_{\substack{ N(\gamma) \le M^{1-b} \\ \gamma \equiv 1 \bmod (1+i)^3
}}
\frac{H_w(\gamma)}{N(\gamma)^{1+2\tau}h_w(\gamma)}\frac{\mu^2_{[i]}(\gamma)
G_w(1,\gamma;\delta_1+w,\delta)G_w(1,\gamma;\delta_2+w,\delta)}
{\prod_{\mu=\pm} \zeta_K(1+\delta_1+w+\mu\delta)\zeta_K(1+\delta_2+w+\mu\delta)}
\\
&+ O\Big( \frac{|\delta_1|^2 }{\log^2 M}
M^{(1-b)(-\tau-\Re(w))} \Big).
\end{split}
\end{align*}

 We apply the last expression above in \eqref{M1innersum} to see that, based on our
choice of ${\mathcal C}$ and $M$, the contribution of the error term above is
\begin{align}
\label{errorM1gammasmall}
\begin{split}
O\Big( \log^2 X |\delta_1|^3 M^{-2\tau(1-b)}\Big).
\end{split}
\end{align}
 On the other hand, the main term coming from the contribution of the terms $N(\gamma) \le M^{1-b}$ to \eqref{M1innersum} equals to
\begin{align}
\label{maintermNgammasmall}
\begin{split}
\frac{1}{2\pi i} \int\limits_{{\mathcal C}}
&\frac{2\pi {\widehat \Phi}(1+w-\tau)}{3\zeta_K(2)\Gamma_\delta(\tau)}
\leg{2^5X}{\pi^2}^{w-\tau}
Z(1+2w;\delta) \Gamma_\delta(w) \frac{2w}{w^2-\tau^2}\eta_\delta(1+2w;1)
\\
& \times \sum_{\substack{ N(\gamma) \le M^{1-b} \\ \gamma \equiv 1 \bmod (1+i)^3
}} \frac{H_w(\gamma)}{N(\gamma)^{1+2\tau}h_w(\gamma)}\frac{\mu^2_{[i]}(\gamma)
G_w(1,\gamma;\delta_1+w,\delta)G_w(1,\gamma;\delta_2+w,\delta)}
{\prod_{\mu=\pm} \zeta_K(1+\delta_1+w+\mu\delta)\zeta_K(1+\delta_2+w+\mu\delta)}
dw.
\end{split}
\end{align}

Note that as $\prod_{\mu =\pm}
\zeta_K(1+\delta_1+w+\mu\delta)^{-1}\zeta_K(1+\delta_2+w
+\mu\delta)^{-1}$ vanishes at $w=-\tau$,  the integrand in \eqref{maintermNgammasmall} has only a simple pole at $w=\tau$ inside ${\mathcal C}$. It then follows from Cauchy's theorem that the expression in \eqref{maintermNgammasmall} equals
\begin{align}
\label{maintermNgammasmallsimplified}
\begin{split}
\frac{2\pi {\widehat \Phi}(1)}{3\zeta_K(2)} \frac{\eta_\delta(1+2\tau;1)}
{\zeta_K(1+2\tau)}
\sum_{\substack{ N(\gamma) \le M^{1-b} \\ \gamma \equiv 1 \bmod (1+i)^3
}} \frac{\mu^2_{[i]}(\gamma)H_\tau(\gamma)}{N(\gamma)^{1+2\tau}h_{\tau}(\gamma)}
G_\tau(1,\gamma;\delta_1+\tau,\delta)G_\tau(1,\gamma;\delta_2+\tau,\delta).
\end{split}
\end{align}

  We apply Lemma \ref{Lemma 3.17} to evaluate \eqref{maintermNgammasmallsimplified} and then combine the result together with the
error term given in \eqref{errorM1gammasmall} to see that the contribution to ${\mathcal S}(1; \Phi){\mathcal W}(\delta_1, \Phi)$ from the part of ${\mathcal M}_1(l)$
restricting to the case $N(\gamma) \le M^{1-b}$ equals
\begin{align}
\label{M1gammasmall}
\frac{2\pi{\widehat \Phi}(1)}{3\zeta_K(2)} ( 1 - M^{-2\tau(1-b)})
+O(\log^2 X |\delta_1|^3 M^{-2\tau(1-b)}).
\end{align}

We now evaluate \eqref{M1sumoverlinHab} for the case $N(\gamma) > M^{1-b}$ by applying Lemma \ref{Lemma 3.16} with $u$, $v$, $g(n)$, $G_w(s,\gamma;u,v)$ as before,
and with $R(x)=P(x)$, $y=M$ this time. We deduce then that for any odd $M^{1-b} <
N(\gamma) <M$,
\begin{align}
\label{sumofgammabig}
\begin{split}
\sum_{\substack{ N(a) \le M/N(\gamma) \\ a \equiv 1 \bmod (1+i)^3 }}
 &\frac{r_{\delta}(a) \mu_{[i]}(a\gamma)}{N(a)^{1+\delta_1+w }
h_w(a)} Q\Big(\frac{\log(M/N(a\gamma))}{\log M}\Big)
= O\Big( \frac{E(\gamma)}{\log^2 M }
\leg{M}{N(\gamma)}^{-\tau -\Re(w)} \exp(-A_0\sqrt{\log (M/N(\gamma)})\Big)
\\
&+\mathop{\text{Res}}_{s=0} \frac{\mu_{[i]}(\gamma) G_w(1+s,\gamma
;\delta_1+w,\delta)}{
s\zeta_K(1+s+\delta_1+w+\delta)\zeta_K(1+s+\delta_1+w-\delta)}
\sum_{k=0}^{\infty} \frac{1}{(s\log M)^k} Q^{(k)} \Big(\frac{\log (M/N(\gamma))}
{\log M}\Big).
\end{split}
\end{align}

  To evaluate the residue above, we write the Taylor expansion of $G_w(1+s,\gamma;u,v)/(\zeta_K(1+s+\delta_1+w
+\delta)\zeta_K(1+s+\delta_1+w-\delta))$ as
$a_0+a_1 s+ a_2 s^2 +\ldots$ to see that $a_0= \frac {4^2}{\pi^2} \cdot (\delta_1+w+\delta)
(\delta_1+w-\delta)G_w(1,\gamma;\delta_1+w,\delta)
 + O((|\delta_1|+|w|)^3)$, $a_1=2 \cdot \frac {4^2}{\pi^2} (\delta_1+w) G_w(1,\gamma;\delta_1+w,\delta)
+ O((|\delta_1|+|w|)^2)$, $a_2= \frac {4^2}{\pi^2} \cdot G_w(1,\gamma;\delta_1+w,\delta)
+O(|\delta_1|+|w|)$, and that $a_n \ll_n 1$
for $n \ge 3$. It follows from this that the residue term in \eqref{sumofgammabig} equals
\begin{align}
\label{sumoveragammabig}
\begin{split}
\frac {4^2}{\pi^2} \cdot \mu_{[i]}(\gamma)G_w(1,\gamma;\delta_1+w,\delta)
\Big(  (\delta_1+w+\delta)(\delta_1+w-\delta)
&Q\Big(\frac{\log (M/N(\gamma))}{\log M}\Big) + 2 \frac{\delta_1+w}{\log M}
Q^{\prime}\Big(\frac{\log (M/N(\gamma))}{\log M}\Big)
\\
&+\frac{1}{\log^2 M}Q^{\prime \prime} \Big(\frac{\log (M/N(\gamma))}{\log M}\Big)
\Big) + O(|\delta_1|^3).
\end{split}
\end{align}

 Again, replacing $a$ by $b$ and $\delta_1$ by $\delta_2$ in \eqref{sumoveragammabig} yields an expression for the sum over $b$ in \eqref{M1sumoverlinHab} when $M^{1-b} <
N(\gamma) <M$.  We can thus apply these expressions to
evaluate \eqref{M1sumoverlinHab} for the case $N(\gamma) > M^{1-b}$. By doing so, we see that the contribution of the remainder terms is
\begin{align}
\label{errorsumoveragammabig}
\begin{split}
&\ll \sum_{M^{1-b} \le N(\gamma) \le M}
\frac{1}{N(\gamma)^{1+2\tau}} \Big( \frac{E(\gamma)}{\log^2 M} |\delta_1|^2
\Big(\frac{M}{N(\gamma)}\Big)^{-\tau-\Re(w)}
\exp(-A_0 \sqrt{\log (M/N(\gamma))}
) + |\delta_1|^5 \Big)
\\
&\ll \frac{|\delta_1|^2}{\log^2 M} M^{-\tau -\Re(w) } + M^{-2\tau(1-b)}
|\delta_1|^5 \log M.
\end{split}
\end{align}

 On the other hand, by writing $Q^{(j)}_\gamma$ for $Q^{(j)}(\log (M/N(\gamma))/\log M)$, we obtain that the main term is
\begin{align*}
\begin{split}
& (\frac {4}{\pi})^4 \cdot \eta_\delta(1+2w;1) \sum_{\substack{ M^{1-b} < N(\gamma) \le M \\
\gamma \equiv 1 \bmod (1+i)^3 }} \frac{\mu^2_{[i]}(\gamma) H_w(\gamma)}{N(\gamma)^{1+2\tau}
h_w(\gamma)} G_w(1,\gamma;\delta_1+w,\delta)G_w(1,\gamma;\delta_2+w,\gamma)
\\
&\times \Big( (\delta_1+w+\delta)(\delta_1+w-\delta)
Q_\gamma + 2 \frac{\delta_1+w}{\log M}
Q^{\prime}_\gamma
+\frac{1}{\log^2 M}Q^{\prime \prime}_\gamma\Big)\\
&\times
\Big((\delta_2+w+\delta)(\delta_2+w-\delta)
Q_\gamma + 2 \frac{\delta_2+w}{\log M}
Q^{\prime}_\gamma
+\frac{1}{\log^2 M}Q^{\prime \prime}_\gamma\Big).
\end{split}
\end{align*}

 We apply Lemma \ref{Lemma 3.17} to compute the expression above. By a suitable change of variables, we see that it equals
\begin{align}
\label{mainsumoveragammabig}
\begin{split}
N(w) + O(M^{-2\tau(1-b)} |\delta_1|^5
\log M),
\end{split}
\end{align}
  where
\begin{align*}
\begin{split}
 N(w)=& (\frac {4}{\pi})^3 \cdot \log M \int_0^b M^{-2\tau(1-x)}
\Big( (\delta_1+w+\delta)(\delta_1+w-\delta)
Q(x) + 2 \frac{\delta_1+w}{\log M}
Q^{\prime}(x)
+\frac{Q^{\prime \prime}(x)}{\log^2 M}\Big)
\\
&\times \Big((\delta_2+w+\delta)(\delta_2+w-\delta)
Q(x) + 2 \frac{\delta_2+w}{\log M}
Q^{\prime}(x)
+\frac{Q^{\prime\prime}(x)}{\log^2 M}\Big) dx.
\end{split}
\end{align*}

  We now apply \eqref{errorsumoveragammabig} and \eqref{mainsumoveragammabig} to evaluate the expression in \eqref{M1innersum} when $N(\gamma) >M^{1-b}$.  Based on our choices for $M$ and ${\mathcal C}$, we see that the error terms in \eqref{errorsumoveragammabig} and \eqref{mainsumoveragammabig} contribute
\begin{align}
\label{errortermofM1contribution}
O\Big( |\delta_1|^3  M^{-2\tau} \log^2 X+ |\delta_1|^6
 M^{-2\tau(1-b) }\log^5 X\Big).
\end{align}
  Moreover, we see that the main term in \eqref{M1innersum} equals
\begin{align}
\label{maintermofM1contribution}
\frac{1}{2\pi i} \int\limits_{\mathcal C} \frac{2\pi {\widehat \Phi}(1+w-\tau)}{3\zeta_K(2)
\Gamma_\delta(\tau)}
\leg{2^5X}{\pi^2}^{w-\tau} Z(1+2w;\delta) \Gamma_\delta(w) N(w) \frac{2w}
{w^2-\tau^2}dw.
\end{align}
Applying the relations
\begin{align*}
 {\widehat \Phi}(1+w-\tau)\Gamma_\delta(w)/\Gamma_\delta(\tau) =&
{\widehat \Phi}(1) +O(|\delta_1|), \\
2w Z(1+2w;\delta)=&(\frac {\pi}{4})^3 \cdot \frac{1}{4(w^2-\delta^2)} + O( |\delta_1| \log^2 X), \\
N(w) \ll & M^{-2\tau(1-b)}|\delta_1|^4 \log M,
\end{align*}
  we see that the expression in \eqref{maintermofM1contribution} equals
\begin{align}
\label{maintermofM1contributionsimplified}
\begin{split}
&  \frac{2\pi {\widehat \Phi}(1)}{3\zeta_K(2)} \frac{1}{2\pi i} \int\limits_{\mathcal C}
\leg{2^5X}{\pi^2}^{w-\tau} (\frac {\pi}{4})^3 N(w) \frac{1}{4(w^2-\delta^2)} \frac{1}{w^2-\tau^2}
dw +O(|\delta_1|^6 M^{-2\tau(1-b)}\log^5 X)\\
=& \frac{2 \pi {\widehat \Phi}(1)}{3\zeta_K(2)} \frac{1}{8\delta_1 \delta_2 \tau}
\Big((\frac {\pi}{4})^3 N(\tau) - \leg{2^5X}{\pi^2}^{-2\tau} (\frac {\pi}{4})^3 N(-\tau)\Big) +
O(|\delta_1|^6 M^{-2\tau(1-b)}\log^5 X).
\end{split}
\end{align}

  Now, using integration by parts together with the observations that $Q(0)=Q^{\prime}(0)=0$,
and $Q(b)=1$, $Q^{\prime}(b)=0$, we obtain after a little calculation that
\begin{align*}
\begin{split}
 (\frac {\pi}{4})^3 N(\tau) =& 8\delta_1\delta_2 \tau M^{-2\tau(1-b)}
+ \frac{4\delta_1 \delta_2}{\log M}
\int_0^b M^{-2\tau(1-x) } \Big|
Q^{\prime}(x) +\frac{Q^{\prime\prime}(x)}{2\delta_1 \log M}
\Big|^2
dx \\
(\frac {\pi}{4})^3 N(-\tau) =& \frac 4{\pi} \cdot  \frac{4\delta_1 \delta_2}{\log M}
\int_0^b M^{-2\tau(1-x) } \Big|
Q^{\prime}(x) +\frac{Q^{\prime\prime}(x)}{2\delta_1 \log M}
\Big|^2 dx.
\end{split}
\end{align*}

 The above expressions allow us to recast \eqref{maintermofM1contributionsimplified} as
\begin{align*}
\begin{split}
\frac{2\pi {\widehat \Phi}(1)}{3\zeta_K(2)} \Big( M^{-2\tau(1-b)}
+ \frac{1-(2^5X/\pi^2)^{-2\tau} }{2\tau\log M}
&\int_0^b M^{-2\tau(1-x) } \Big|
Q^{\prime}(x) +\frac{Q^{\prime\prime}(x)}{2\delta_1 \log M}
\Big|^2
dx \Big)
+O(M^{-2\tau(1-b)}
|\delta_1|^6 \log^5 X).
\end{split}
\end{align*}

 It follows from this and \eqref{errortermofM1contribution} that the contribution of ${\mathcal M}_1(l)$ to ${\mathcal S}(1; \Phi){\mathcal W}(\delta_1, \Phi)$ from the terms $M^{1-b}
< N(\gamma) \le M$ equals
\begin{align*}
\begin{split}
& \frac{2\pi{\widehat \Phi}(1)}{3\zeta_K(2)} \Big( M^{-2\tau(1-b)}
+ \frac{1-(2^5X/\pi^2)^{-2\tau} }{2\tau\log M}
\int_0^b M^{-2\tau(1-x) } \Big|
Q^{\prime}(x) +\frac{Q^{\prime\prime}(x)}{2\delta_1 \log M}
\Big|^2
dx \Big)
\\
+&O(|\delta_1|^3 M^{-2\tau} \log^2 X +
M^{-2\tau(1-b)} |\delta_1|^6 \log^5 X).
\end{split}
\end{align*}

 Combining the above with \eqref{M1gammasmall}, we conclude that the ${\mathcal M_1}(l)$
contribution to ${\mathcal S}(1; \Phi){\mathcal W}(\delta_1, \Phi)$  equals
\begin{align}
\label{M1contribution}
\begin{split}
& \frac{2\pi{\widehat \Phi}(1)}{3\zeta_K(2)} \Big( 1+
\frac{1-(2^5X/\pi^2)^{-2\tau} }{2\tau\log M}  \int_0^b M^{-2\tau(1-x) } \Big|
Q^{\prime}(x) +\frac{Q^{\prime\prime}(x)}{2\delta_1 \log M}
\Big|^2
dx \Big)+ O(M^{-2\tau(1-b)} |\delta_1|^6 \log^5 X).
\end{split}
\end{align}

  We obtain the ${\mathcal M}_2(l)$ contribution to \eqref{WintermofM} in
a similar way to see that it equals
\begin{align}
\label{M2contribution}
\begin{split}
&  -\frac{2\pi{\widehat \Phi}(1)}{3\zeta_K(2)} \leg{2^5X}{\pi^2}^{-\tau}
\frac{(2^5X/\pi^2)^{\delta}-(2^5X/\pi^2)^{-\delta}}{2\delta \log M}
\int_{0}^{b} M^{-2\tau(1-x) } \Big|
Q^{\prime}(x) +\frac{Q^{\prime\prime}(x)}{2\delta_1 \log M}
\Big|^2
dx\\
 +& O(X^{-\tau} M^{-2\tau(1-b)} |\delta_1|^6 \log^5 X).
\end{split}
\end{align}

 By applying \eqref{S1}, \eqref{M1contribution} and \eqref{M2contribution} in \eqref{WintermofM}, we deduce readily the statement of Proposition \ref{Proposition 2.4}.

\subsection{Completion of the proof}
  We are now able to complete our proof of Theorem \ref{thm: nonvanishing}. For this, we take our parameters exactly as those used by Conrey and Soundararajan in their proof of \cite[Theorem 1]{C&S}. Precisely, we take $\kappa=10^{-10}, M=X^{\frac 12
-5\kappa}$ in \eqref{mollifier} and
$S=\pi/(2(1-b))(1-20\kappa))$ with $b=0.64$. We further take $R=6.8$ and $P(x) =3 (x/b)^2 -2 (x/b)^3$.  Let $\sigma_0$ be given as in Proposition \ref{Proposition 2.1} and let $\Phi$ to be a smooth function supported
in $(1,2)$ satisfying $0 \leq \Phi(t) \leq 1$ for all $t$, $\Phi(t)=1$ for
$t\in (1+\epsilon,2-\epsilon)$, and $|\Phi^{(\nu)}(t)| \ll_{\nu, \varepsilon}
1$.    We apply Proposition \ref{Proposition 2.1} in \eqref{Nbound} to deduce that
\begin{align}
\label{6.30}
{\mathcal N}(X,\Phi) \le \frac{X{\mathcal S}(1;\Phi)}{8S \sinh (\frac{\pi R}{2S})}
(J_1(X;\Phi) + J_2(X;\Phi)) +o(X),
\end{align}
where $J_1$ and $J_2$ are given in \eqref{Jdef}.

 Now, Proposition \ref{Proposition 2.4}  implies that for real numbers $u, v$ satisfying
$\kappa \log X \ge |u + iv| \ge -1/\epsilon$ and $u \ge -1/\epsilon$, we have
$$
{\mathcal W}\Big(\frac{u+iv}{\log X},\Phi\Big)
= {\mathcal V}(u,v)
+ O\Big(
 M^{-2u(1-b)/\log X} \frac{(1+|u|+|v|)^6}{\log X}
\Big),
$$
where
$$
{\mathcal V}(u,v)= 1+ \frac{e^{-u}\log X}{2\log M} \Big( \frac{\sinh u}{u}
-\frac{\sin v}{v} \Big) \int_0^b M^{-2u(1-x)/\log X} \Big| Q^{\prime}(x)
+ \frac{Q^{\prime\prime}(x) \log X}{2 (x+iy)\log M} \Big|^2 dx.
$$
 The above expression implies that ${\mathcal V}(u,v) \ge 1$, from which
we deduce that
\begin{align}
\label{6.31}
J_1(X;\Phi) = \int_0^S \cos \leg {\pi t}{2S}
\log {\mathcal V}(-R,t) dt + O\Big( \frac{1}{\log X}\Big).
\end{align}
Moreover, based on Proposition \ref{Proposition 2.1} and our choice for $S$,  we can extend the integral in the definition of $J_2(X; \Phi)$ in \eqref{Jdef} to infinity with an negligible error. This way,  we obtain
that
\begin{align}
\label{6.32}
J_2(X; \Phi) = \int_0^{\infty} \sinh \leg {\pi u}{2S}
\log {\mathcal V}(u-R, S) du + o(1).
\end{align}
  We conclude from \eqref{6.30}-\eqref{6.32} that
\begin{align*}
{\mathcal N}(X,\Phi) &\le \frac{X{\mathcal S}(1;\Phi)}{8S \sinh \leg {\pi R}{2S} }
\Big( \int_0^S \cos \leg {\pi t}{2S}
\log {\mathcal V}(-R,t) dt +  \int_0^{\infty} \sinh \leg {\pi u}{2S}
\log {\mathcal V}(u-R, S) du\Big) + o(X).
\end{align*}
 As shown in \cite[p. 10]{C&S}, the above further implies that
${\mathcal N}(X,\Phi) \le 0.79 X {\mathcal S}(1;\Phi) + o(X)$. By summing over $X=x/2$, $x/4$, $\ldots$, we obtain the proof of Theorem \ref{thm: nonvanishing}.

\vspace*{.5cm}

\noindent{\bf Acknowledgments.} P. G. is supported in part by NSFC grant 11871082.

\bibliography{biblio}
\bibliographystyle{amsxport}

\vspace*{.5cm}

\end{document}